\documentclass[10pt,final]{siamltex}
\usepackage{amsmath}
\usepackage{CJKutf8}
\usepackage{bm}
\usepackage{amssymb,version}
\usepackage{cases}
\usepackage{color}
\usepackage{verbatim}
\usepackage{fancyvrb}
\usepackage{graphics}
\usepackage{epsfig}
\usepackage{epstopdf}
\usepackage{enumerate}
\usepackage{lipsum}
\usepackage{multirow}
\usepackage{mathrsfs}

\usepackage{graphicx}
\usepackage{subfig}
\usepackage{svg}
\usepackage{todonotes}

\allowdisplaybreaks

\newcommand{\normmm}[1]{{\left\vert\kern-0.25ex\left\vert\kern-0.25ex\left\vert #1
    \right\vert\kern-0.25ex\right\vert\kern-0.25ex\right\vert}}

\begin{document}
	\newtheorem{thm}{Theorem}
	\newtheorem{lem}[thm]{Lemma}
    \title{Robust superconvergence analysis of physics-preserving RMAC scheme for the Stokes and Navier--Stokes equations on non-uniform grids at high Reynolds numbers
\thanks{This work is supported in part by the National Natural Science Foundation of China (Grant Nos.12271302, 12131014) and Shandong Provincial Natural Science Foundation for Outstanding Youth Scholar (Grant No. ZR2024JQ030). Authors are sorted alphabetically.}}

    \author{ Binghong Li \thanks{School of Mathematics, Shandong University, Jinan, Shandong, 250100, P.R. China. Email: binghongsdu@163.com}.
        \and Xiaoli Li \thanks{Corresponding Author. School of Mathematics, Shandong University, Jinan, Shandong, 250100, P.R. China. Email: xiaolimath@sdu.edu.cn}.
        \and Xu Li \thanks{Eastern Institute for Advanced Study, Eastern Institute of Technology,
		  Ningbo, Zhejiang, 315200, P. R. China. Email: xulisdu@126.com}.
        \and Hongxing Rui \thanks{ School of Mathematics, Shandong University, Jinan, Shandong, 250100, P.R. China. Email: hxrui@sdu.edu.cn}.
}
\graphicspath{{figures/},}
\maketitle

\begin{abstract}
The velocity errors of the classical marker and cell (MAC) scheme are dependent 
on the pressure approximation errors, which is non-pressure-robust and will cause the accuracy of the velocity approximation to deteriorate 
when the pressure approximation is poor.
In this paper, we first propose the reconstructed MAC scheme (RMAC) based on the finite volume method to obtain the 
pressure-robustness for the time-dependent Stokes equations
and then construct the $\mu$-robust and physics-preserving RMAC 
scheme on non-uniform grids for the Navier--Stokes equations, where $\mu$-robustness means that the velocity errors do not blow up for small viscosity $\mu$ when the true velocity is sufficiently smooth. 
Compared with the original MAC scheme, which was analyzed in \textit{[SIAM J. Numer. Anal. 55 (2017): 1135-1158]}, 
the RMAC scheme is different only on the right-hand side for Stokes equations. 
It can also be proved that the constructed scheme satisfies the local mass conservation law, the discrete unconditional energy dissipation law, 
the momentum conservation, and the angular momentum conservation for the Stokes and Navier--Stokes equations. Furthermore, 
by constructing the new auxiliary function depending on the velocity and using the high-order consistency analysis, 
we can obtain the pressure-robust and $\mu$-robust error estimates for the velocity and derive the second-order superconvergence  for the velocity and pressure in the discrete $l^{\infty}(l^2)$ norm 
on non-uniform grids and the discrete $l^{\infty}(l^{\infty})$ norm on uniform grids. 
Finally, numerical experiments using the constructed schemes are demonstrated to show the robustness for our constructed schemes.

\end{abstract}

 \begin{keywords}
MAC scheme, Stokes equations, pressure-robust, physics-preserving, superconvergence
 \end{keywords}
   \begin{AMS}
65N06, 65N08, 65N12, 65N15, 62G35
    \end{AMS}

\section{Introduction}
The Stokes and Navier--Stokes equations describe the motion of incompressible viscous flow and constitute cornerstone models in continuum mechanics. Given the inherent challenges in obtaining analytical solutions, numerical methods are often used to solve the Stokes and Navier--Stokes equations (cf. \cite{han2023analysis, qiu2024h}). There are many methods in this field, including the finite difference method, the finite volume method, and the finite element method (see, e.g., \cite{girault1979finite, glowinski2003finite, li2022finite, li2022new}
and the references therein). Among these methods, one of the best and simplest is the marker-and-cell (MAC) method, which was proposed in Girault and Lopez \cite{girault1996finite}. Han and Wu have also acknowledged in \cite{han1998new} that the MAC scheme is one of the simplest and most effective numerical schemes for solving the Stokes equations and the Navier--Stokes equations. The MAC scheme enforces the incompressibility constraint of the velocity field at the discrete level while ensuring exact local conservation of mass, momentum, and kinetic energy.

There have been many theoretical error analyses of MAC schemes for the Stokes and Navier--Stokes equations. It was proved in \cite{girault1996finite,han1998new,kanschat2008divergence,Nicolaides1992Analysis} that the MAC scheme achieved first-order convergence for both velocity in the $H^1$ norm and pressure in the $L^2$ norm on uniform grids. However, numerical examples in \cite{Nicolaides1992Analysis} demonstrate that both velocity and pressure exhibit second-order accuracy. In 2015, Li and Sun \cite{li2015superconvergence} investigated the second-order convergence of the velocity in $L^2$ norm for stationary Stokes equations on non-uniform grids under the unproven assumption that the pressure exhibits second order convergence. To fill this gap,  Li and Rui \cite{Rui2016} established the second-order superconvergence of both velocity and pressure for stationary Stokes equations on non-uniform grids by innovative construction of an auxiliary function depending on the velocity and discretizing parameters, with subsequent extensions to time-independent Stokes equations in \cite{LI20181499}. As a prominent spatial discretization method, the MAC scheme has been widely adopted across diverse fluid models, particularly demonstrating high efficiency in complex Navier-Stokes systems. In
\cite{li2018superconvergence,li2020error,li2022superconvergence}, MAC schemes for Navier--Stokes equations were developed, achieving the second-order superconvergence in the $L^2$ norm for both velocity and pressure. In addition, Li and Shen \cite{li2020sav} established second-order error estimates using MAC scheme in spatial discretization for the Cahn--Hilliard--Navier--Stokes phase-field model. Furthermore, Dong and Ying \cite{dong2023second} proposed the modified MAC scheme for Stokes interface problems, achieving second-order accuracy on uniform grids. 

We would like to remark that in all aforementioned proofs,  the velocity error estimation remains dependent on the exact pressure solution, which makes the results not robust to the pressure.
The pressure-robustness in mixed finite elements is thoroughly discussed in \cite{John2,PL1}, with emphasis on its critical importance for discrete methods. 
It is pointed out in, e.g., \cite{PL1,LM:2016} that an improved discretization of the right-hand side of the momentum equation can restore pressure-robustness for the classical non-pressure-robust mixed finite element methods such as
Taylor--Hood elements and Bernardi--Raugel elements.
Besides, it is proved that the MAC scheme can be naturally transformed into a finite volume method and a mixed finite element method (cf. \cite{han1998new,1996analysis}).This motivates our investigation into whether there is such an improved discretization for the MAC scheme. 

As we all know, finite element methods, which preserve exact enforcement of divergence-free condition, are pressure-robust(see, e.g., \cite{John2}). In contrast, while the MAC scheme rigorously enforces discrete divergence-free conditions for numerical velocity (matching the continuous divergence-free constraint of exact solutions), the core analytical challenge stems from the inherent non-equivalence between discrete divergence-free conditions and their continuous counterparts. This discrepancy induces pressure non-robustness and leads to accuracy degradation in the error estimates for the $\mu$-robustness (cf. \cite{garcia2021convergence, ahmed2018really}). Consequently, the pressure-robustness theories developed for finite elements cannot be directly generalized to MAC schemes, as the latter's inherent non-pressure-robust property stems from this discrete-continuous operator mismatch.
In addition, the introduction of a velocity-dependent auxiliary function (as demonstrated in \cite{Rui2016,li2018superconvergence}) becomes essential to attain superconvergence properties on non-uniform grids. However, such auxiliary functions inherently introduce fundamental challenges to the proof of $\mu$-robustness. The key innovation of this work lies in addressing pressure-robust and $\mu$-robust error estimates while preserving second-order superconvergence  for the velocity and pressure on non-uniform grids. Our main contributions in this paper include the following parts:
\begin{itemize}
	\item Based on the finite volume method, we propose the reconstructed MAC scheme (RMAC) for the Stokes and Navier–Stokes equations and prove that the constructed scheme preserves an important invariance property of the continuous model: the velocity solution remains unchanged when an arbitrary gradient field is added to the right-hand side of the momentum equation. Inspired by this result, we construct the new splitting method to decompose the original model into two sub-problems to obtain the pressure-robust estimate for the velocity.
	\item The constructed RMAC scheme can satisfy local mass conservation law, the discrete unconditional energy dissipation law, momentum conservation, and angular momentum conservation for the Stokes and Navier--Stokes equations. 
	\item  We can obtain the pressure-robust and $\mu$-robust error estimate for the velocity and derive the second-order superconvergence for the velocity and pressure in the discrete $l^{\infty}(l^2)$ norm on non-uniform grids and the discrete $l^{\infty}(l^{\infty})$ norm on uniform grids by constructing the new auxiliary function depending on the velocity and using the high-order consistency analysis.
	
\end{itemize}

The paper is organized as follows. In Section 2 we give the model along with some preliminaries. In Section 3 we propose the RMAC scheme for the time-dependent Stokes equations to achieve pressure-robustness, establishing its stability properties. In Section 4, we demonstrate rigorous error analysis of both vilocity and pressure for $\mu$-robustness and the second-order supercovergence. In Section 5 we show some discrete conservation properties for the Stokes and Navier--Stokes equations. In Section 6 we perform numerical experiments comparing RMAC with classical MAC schemes, validating pressure-robustness and $\mu$-robustness of velocity $\mathbf{u}$ along with pressure $p$ approximations. The convergence rates are in agreement with the theoretical analysis. Technical details supporting the error estimates are provided in the Appendix.

Throughout the paper we use $C$, with or without subscript, to denote a positive constant, which could have different values at different appearances.
\section{The Problem and Some Preliminaries}
In this section, we first consider the time-dependent Stokes equations for an incompressible fluid in a two-dimensional spatial domain, subject to homogeneous boundary conditions and a homogeneous initial condition. For non-zero initial data scenarios, the problem can be reformulated by incorporating the initial data into the source term, thereby achieving a homogeneous initial condition.

Find the pressure $p$ and the velocity vector $\mathbf{u}=(u^x,u^y)$ such that
\begin{equation}\label{e1}
\left\{
\begin{array}{l}
\displaystyle\frac{\partial \mathbf{u}}{\partial t}-\mu \Delta \mathbf{u}+\nabla p=\mathbf{g},~~~~(x,y,t) \in \Omega\times J,\\
\displaystyle\nabla\cdot \mathbf{u}=0,~~~~~~~~~~~~~~~~~~(x,y,t) \in\Omega\times J,\\
\displaystyle\mathbf{u}=0,~~~~~~~~~~~~~~~~~~~~~~(x,y,t) \in \partial \Omega\times J,\\
\mathbf{u}(x,y,0) = 0,~~~~~~(x,y) \in \Omega.
\end{array}
\right.
\end{equation}
Here $\mu>0$, represents the viscosity coefficient of the fluid. $\mathbf{g}=(g^x,g^y)\in (L^2(\Omega))^2$, represents the source term. For simplicity, we take $\Omega=(0,L_x)\times(0,L_y)$ as a rectangular domain.  $J=(0,T]$, and T denotes the final time.

Firstly, we give the partitions and notations as follows. For $n\in \mathbb{N}\bigcup \{0\}$, we set $\Delta t = T/N$ and define $t^n = n\Delta t$. The two dimensional domain $\Omega$ is partitioned by $\delta_{x}\times\delta_{y}$, where
\begin{equation*}
	\aligned
	\delta_{x}:0=x_{0}<x_{1}<\cdot\cdot\cdot<x_{N_{x}-1}<x_{N_{x}}=L_x,~~\delta_{y}:0=y_{0}<y_{1}<\cdot\cdot\cdot<y_{N_{y}-1}<y_{N_{y}}=L_y,
	\endaligned
\end{equation*}
and any coordinates out of the physical domain are treated as their neighboring values.

For $i=1,\cdot\cdot\cdot,N_x$ and $j=1,\cdot\cdot\cdot,N_y$,
define
\begin{equation*}
\aligned
&x_{i-1/2}=\frac{x_{i}+x_{i-1}}{2},~h_{i-1/2}=x_{i}-x_{i-1},~ h=\max\limits_{i}{h_{i-1/2}},\\
&y_{j-1/2}=\frac{y_{j}+y_{j-1}}{2},~k_{j-1/2}=y_{j}-y_{j-1},~ k=\max\limits_{j}{k_{j-1/2}},\\
&h_{i}=\frac{h_{i+1/2}+h_{i-1/2}}{2},~k_{j}=\frac{k_{j+1/2}+k_{j-1/2}}{2},\\
&\Omega_{i+1/2,j+1/2}=(x_i,x_{i+1})\times (y_j,y_{j+1}).
\endaligned
\end{equation*}
It is clear that
$$h_0=\frac{h_{1/2}}{2},~h_{N_x}=\frac{h_{N_x-1/2}}{2},~k_0=\frac{k_{1/2}}{2},~k_{N_y}=\frac{k_{N_y-1/2}}{2}.$$
We suppose the partition is regular, which {\color{black}means that} there is a positive constant $C_0$ such that
\begin{equation}\label{e2}
\min\limits_{i,j}\{h_{i+1/2},k_{j+1/2}\}\geq C_0 \max\limits_{i,j}\{h_{i+1/2},k_{j+1/2}\}.
\end{equation}

For a function $\phi(x,y,t)$, let $\phi_{l,m}^n$ denote $\phi(x_l,y_m,t^n)$, where $l$ may take values $i,~i+1/2$ for integer $i$, and $m$ may take values $j,~j+1/2$ for integer $j$. For discrete functions, define
\begin{equation}\label{e3}
\left\{
\begin{array}{lll}
[d_x\phi]_{i+1/2,m}^n=\displaystyle\frac{\phi_{i+1,m}^n-\phi_{i,m}^n}{h_{i+1/2}},
\quad[D_y\phi]_{l,j+1}^n=\displaystyle\frac{\phi_{l,j+3/2}^n-\phi_{l,j+1/2}^n}{k_{j+1}},\\

[d_y\phi]_{l,j+1/2}^n=\displaystyle\frac{\phi_{l,j+1}^n-\phi_{l,j}^n}{k_{j+1/2}},~~~
\quad[D_x\phi]_{i+1,m}^n=\displaystyle\frac{\phi_{i+3/2,m}^n-\phi_{i+1/2,m}^n}{h_{i+1}},\\

[d_{t}\phi]_{l,m}^n=\displaystyle\frac{\phi_{l,m}^{n}-\phi_{l,m}^{n-1}}{\Delta t}.
\end{array}
\right.
\end{equation}

For functions $f(x,y),g(x,y),$ define $L^{2}$ inner products and norms

$$(f,g)=\int_{\Omega}f(x,y)g(x,y)dxdy,\quad||f||_{L^2}=\sqrt{(f,f)}.$$

Define the discrete $l^2$ inner products and norms as follows,\\
\begin{equation*}
\aligned
&(f,g)_{l^2,M}=\sum\limits_{i=0}^{N_{x}-1}\sum\limits_{j=0}^{N_{y}-1}h_{i+1/2}k_{j+1/2}f_{i+1/2,j+1/2}g_{i+1/2,j+1/2},\\
&(f,g)_{l^2,T_x}=\sum\limits_{i=0}^{N_{x}}\sum\limits_{j=1}^{N_{y}-1}h_{i}k_{j}f_{i,j}g_{i,j},~(f,g)_{l^2,T_y}=\sum\limits_{i=1}^{N_{x}-1}\sum\limits_{j=0}^{N_{y}}h_{i}k_{j}f_{i,j}g_{i,j},\\
&~~~~~~~||f||_{l^2,\xi}=(f,f)_{l^2,\xi},\quad \xi=M, T_x, T_y.
\endaligned
\end{equation*}
Moreover, we define
\begin{equation*}
\aligned
&(f,g)_{l^2,T,M}=\sum\limits_{i=1}^{N_{x}-1}\sum\limits_{j=0}^{N_{y}-1}h_{i}k_{j+1/2}f_{i,j+1/2}g_{i,j+1/2},\\
&(f,g)_{l^2,M,T}=\sum\limits_{i=0}^{N_{x}-1}\sum\limits_{j=1}^{N_{y}-1}h_{i+1/2}k_{j}f_{i+1/2,j}g_{i+1/2,j},\\
&||f||_{l^2,T,M}=(f,f)_{l^2,T,M},\quad ||f||_{l^2,M,T}=(f,f)_{l^2,M,T}.
\endaligned
\end{equation*}
Finally define the discrete $H^1$ norm and discrete $l^2$ norm of a vectored-valued function $\textbf{u}$,
\begin{align}
 \|D\mathbf{u}\|_{l^2}^2&=\|d_xu^x\|_{l^2,M}^2+\|D_yu^x\|_{l^2,T_y}^2+\|D_xu^y\|_{l^2,T_x}^2+\|d_yu^y\|_{l^2,M}^2.
\label{e4}\\
\|\mathbf{u}\|_{l^2}^2&=\|u^x\|_{l^2,T,M}^2+\|u^y\|_{l^2,M,T}^2.\label{e5}
\end{align}

\section{Pressure-Robust Reformulation and the Stability}
In this section, we formally present the concept of pressure robustness and $\mu$-robustness and reformulate the original model into an equivalent weak formulation, essential to ensure pressure-robustness error estimates. Subsequently, we establish the RMAC scheme and prove its energy stability through rigorous analysis.
\subsection{Pressure Robustness and {$\mu$-robustness}}
 The velocity $\mathbf{u}$ of the time-independent Stokes equation has an important invariance related to the gradient field: Under the premise that the boundary conditions do not involve pressure, adding a gradient field force $\nabla \phi \ (\phi \in H^1(\Omega))$ to the right side of the first equation of (\ref{e1}) will only cause changes in pressure rather than velocity. This property can be succinctly described by the following equation:
\begin{equation}\label{e24}
\mathbf{g} \to \mathbf{g} + \nabla \phi \Rightarrow (\mathbf{u},p) \to (\mathbf{u},p+\phi).
\end{equation}
If the discrete method can preserve this invariance well, it can be proved that the velocity approximation error solved by this method does not depend on pressure, which is called pressure robustness \cite{John2}.

However, the result we have obtained in \cite{LI20181499} is:
\begin{equation}\label{e23}
\aligned
\displaystyle
\|\mathbf{W}^{m}-\mathbf{u}^{m}\|_{l^2}\leq C(\frac{1}{\mu},\mu)(h^2 + k^2 + \Delta t)(\|\mathbf{u}\|_{L^{\infty}(J;W^{4,\infty}(\Omega))} + \|p\|_{L^{\infty}(J;W^{3,\infty}(\Omega))}),
\endaligned
\end{equation}
for any $m \geq 1$. Our analysis reveals that the velocity error estimates exhibit dependence on the exact pressure solution. Furthermore, these estimates demonstrate direct proportionality to $1/\mu$, resulting in significant error magnification for $\mu \ll 1$. For the Navier--Stokes equations, eliminating the $1/\mu$ error scaling yields enhanced accuracy in high-Reynolds-number regimes (see \cite{GJN21} and the references therein), a property alternatively termed $Re$-semi-robustness or convection-robustness in contemporary literature.
Therefore, the error estimate of velocity that does not depend on pressure and directly on $1/\mu$, is the goal of this work.
\subsection{The RMAC scheme}
The model (\ref{e1}) can be transformed into the following classical one directly:
\begin{equation}\label{e20}
\left\{
\begin{array}{l}
\displaystyle\frac{\partial {u}^x}{\partial t}-\mu \frac{\partial^2{u}^x}{\partial x^2}-
\mu \frac{\partial^2{u}^x}{\partial y^2}+\frac{\partial p}{\partial x}=g^x,~~~~(x,y,t) \in \Omega\times J,\\
\displaystyle\frac{\partial {u}^y}{\partial t}-\mu \frac{\partial^2{u}^y}{\partial x^2}-
\mu \frac{\partial^2{u}^y}{\partial y^2}+\frac{\partial p}{\partial y}=g^y,~~~~(x,y,t) \in \Omega\times J,\\
\displaystyle \frac{\partial {u}^x}{\partial x}+\frac{\partial {u}^y}{\partial y}=0,~~~~~~~~~~~~~~~~~~~~~~~~~~(x,y,t) \in \Omega\times J,\\
\displaystyle{\mathbf{u}}=0,~~~~~~~~~~~~~~~~~~~~~~~~~~~~~~~~~~~~(x,y,t) \in \partial \Omega\times J,\\
\displaystyle{\mathbf{u}}(x,y,0) = 0,~~~~~~~~~~~~~~~~~~~~~(x,y,t) \in \Omega\times J.
\end{array}
\right.
\end{equation}

In order to eliminate the error in the discrete equation, we define the following average operators $\mathscr{A}^x$ and $ \mathscr{A}^y$ based on the finite volume method, such that
\begin{equation*}
\aligned
&\mathscr{A}^x = \frac{1}{h_i} \int_{x_{i-1/2}}^{x_{i+1/2}} dx, \  \ 
\mathscr{A}^y = \frac{1}{k_j} \int_{y_{j-1/2}}^{y_{j+1/2}} dy. \\
\endaligned
\end{equation*}
By applying the average operators $\mathscr{A}^x$ and $\mathscr{A}^y$  to the first two equations in (\ref{e20}) at $y = y_{j+1/2}$ and $x = x_{i+1/2}$, respectively, we can get the reformulation as follows:
\begin{align}
&\frac{1}{h_i} \int_{x_{i-1/2}}^{x_{i+1/2}} \frac{\partial u^x}{\partial t}(x,y_{j+1/2},t^n) dx - \frac{\mu}{h_i} \int_{x_{i-1/2}}^{x_{i+1/2}} \left(\frac{\partial^2 u^x}{\partial x^2}(x,y_{j+1/2},t^n) + \frac{\partial^2 u^x}{\partial y^2}(x,y_{j+1/2},t^n)\right) dx\notag\\
&~~~~~~~~~~~~~~~~~~~~~~~~~~~~~+D_x p^n_{i,j+1/2} = f^{x,n}_{i,j+1/2}, \ \ 1\leq i\leq N_x-1,0\leq j\leq N_y-1,\label{e21} \\
&\frac{1}{k_i} \int_{y_{j-1/2}}^{y_{j+1/2}} \frac{\partial u^y}{\partial t}(x_{i+1/2},y,t^n) dy - \frac{\mu}{k_i} \int_{y_{j-1/2}}^{y_{j+1/2}} \left(\frac{\partial^2 u^y}{\partial x^2}(x_{i+1/2},y,t^n) + \frac{\partial^2 u^y}{\partial y^2}(x_{i+1/2},y,t^n)\right)dy\notag\\
&~~~~~~~~~~~~~~~~~~~~~~~~~~~~~+D_y p^n_{i+1/2,j} = f^{y,n}_{i+1/2,j}, \ \ 0\leq i\leq N_x-1,1\leq j\leq N_y-1,\label{e22} 
\end{align}
where 
\begin{equation*}
\aligned
&f_{i,j+\frac{1}{2}}^{x,n} = \frac{1}{h_i} \int_{x_{i-1/2}}^{x_{i+1/2}} g^x(x,y_{j+1/2},t^n)dx,~f_{i+1/2,j}^{y,n} = \frac{1}{k_j} \int_{y_{j-1/2}}^{y_{j+1/2}} g^y(x_{i+1/2},y,t^n) dy.
\endaligned
\end{equation*}

Let $\{W^{x,n}_{i,j+1/2}\},~\{W^{y,n}_{i+1/2,j}\}$ and $\{Z^{n}_{i+1/2,j+1/2}\}$ denote the discrete approximations of $\{{u}^{x,n}_{i,j+1/2}\}, ~\{{u}^{y,n}_{i+1/2,j}\}$ and $\{p^{n}_{i+1/2,j+1/2}\}$, respectively, and set the boundary and initial value as follows:
\begin{equation}\label{e6}
\left\{
\begin{array}{l}
\displaystyle W_{0,j+1/2}^{x,n}=W_{N_x,j+1/2}^{x,n}=0,\quad 0\leq j\leq N_y-1,\\
\displaystyle W_{i,0}^{x,n}=W_{i,N_y}^{x,n}=0,~~~~~~~~~\quad 0\leq i\leq N_x,\\
\displaystyle W_{0,j}^{y,n}=W_{N_x,j}^{y,n}=0,~~~~~~~~~\quad 0\leq j\leq N_y,\\
\displaystyle W_{i+1/2,0}^{y,n}=W_{i+1/2,N_y}^{y,n}=0,~\quad 0\leq i\leq N_x-1,\\
W^{x,0}_{i,j+1/2} = 0,\quad 1\leq i\leq N_x-1,\quad 0\leq j\leq N_y - 1,\\
W^{y,0}_{i+1/2,j} = 0,\quad 0\leq i\leq N_x-1,\quad 1\leq j\leq N_y - 1.
\end{array}
\right.
\end{equation}

We find $\{W^{x,n}_{i,j+1/2}\},~\{W^{y,n}_{i+1/2,j}\}$ and $\{Z^{n}_{i+1/2,j+1/2}\}$ for $n \geq 1$ such that, 
\begin{align}
&d_{t}{W}^{x,n}_{i,j+1/2}-\mu \frac{d_xW^{x,n}_{i+1/2,j+1/2}-d_xW^{x,n}_{i-1/2,j+1/2}}{h_i}
-\mu \frac{D_yW^{x,n}_{i,j+1}-D_yW^{x,n}_{i,j}}{k_{j+1/2}}\notag\\
&~~~~~+D_xZ_{i,j+1/2}^n=f_{i,j+1/2}^{x,n},\ \ 1\leq i\leq N_x-1,0\leq j\leq N_y-1,\label{e7}\\
&d_{t}{W}^{y,n}_{i+1/2,j}-\mu \frac{D_xW^{y,n}_{i+1,j}-D_xW^{y,n}_{i,j}}{h_{i+1/2}}
-\mu \frac{d_yW^{y,n}_{i+1/2,j+1/2}-d_yW^{y,n}_{i+1/2,j-1/2}}{k_{j}}\notag\\
&~~~~~+D_yZ_{i+1/2,j}^n=f_{i+1/2,j}^{y,n},\ \ 0\leq i\leq N_x-1,1\leq j\leq N_y-1,\label{e8}\\
&d_xW^{x,n}_{i+1/2,j+1/2}+d_yW^{y,n}_{i+1/2,j+1/2}\notag\\
&~~~~~=0,\ \ 0\leq i\leq N_x-1,0\leq j\leq N_y-1.\label{e9}
\end{align}

Note that $D_x$ and $D_y$ above are exactly the difference operators acting on the pressure in the RMAC scheme. 
This inspires us to use $f_{i,j+\frac{1}{2}}^{x,n}$ and $f_{i+1/2,j}^{y,n}$ as discretizations of the right-hand side (see \eqref{e7} and \eqref{e8}).

Define 
\begin{align*}
\textbf{V}=H^1_0(\Omega)\times H^1_0(\Omega),  \qquad W=\left\{q\in L^2(\Omega): \int_\Omega qdx=0\right\}.
\end{align*}
Corresponding to the partition $\delta_{x} \times \delta_{y}$, we define $W_h$ as $\mathcal{P}_0$ finite element space with discrete zero-mean and $\mathbf{V_h}$ as $\mathcal{P}_1$ finite element space with zero-boundary, see \cite{Rui2016} for details. Define the operator $P_h: W \rightarrow W_h$ as
\begin{align*} 
	(P_h q)_{i+1/2,j+1/2} = q_{i+1/2,j+1/2} - (q, 1)_{l^2,M}/|\Omega|.
\end{align*}
We now prove that scheme \eqref{e7}-\eqref{e9} satisfies the following discrete invariance property.
\medskip
\begin{lem}\label{lem:discrete_invariance}
	Let $\{W^{x,n}_{i,j+1/2}\},~\{W^{y,n}_{i+1/2,j}\}$ and $\{Z^{n}_{i+1/2,j+1/2}\}$ be the solutions of \eqref{e7}-\eqref{e9}. The discrete velocity 
	satisfies the following invariance property:
	\begin{equation}\label{eq:discrte_invariance}
		\begin{aligned}
		&\mathbf{g} \to \mathbf{g} + \nabla \phi \Rightarrow \\
		&(\{W^{x,n}_{i,j+1/2}\},\{W^{y,n}_{i+1/2,j}\},\{Z^{n}_{i+1/2,j+1/2}\}) \to (\{W^{x,n}_{i,j+1/2}\},\{W^{y,n}_{i+1/2,j}\},\{(Z+P_h \phi)_{i+1/2,j+1/2}^n\}),
		\end{aligned}
	\end{equation}
	where $\mathbf{g} = (f_{i,j+\frac{1}{2}}^{x,n}, f_{i+1/2,j}^{y,n} ) =  (\frac{1}{h_i} \int_{x_{i-1/2}}^{x_{i+1/2}} g^x(x,y_{j+1/2},t^n)dx,  \frac{1}{k_j} \int_{y_{j-1/2}}^{y_{j+1/2}} g^y(x_{i+1/2},y,t^n) dy)$.
\end{lem}
\begin{proof}
	Let 
	\begin{align*} 
		c=(\phi(\cdot.\cdot, t^n),1)_{l^2,M}/|\Omega|
	\end{align*}
	be the discrete mean of $\phi(\cdot.\cdot, t^n)$, which may not be zero. Simple calculations yield
	\begin{align*}
		 \frac{1}{h_i} \int_{x_{i-1/2}}^{x_{i+1/2}} \frac{\partial \phi}{\partial x}(x,y_{j+1/2},t^n)dx = \frac{(\phi_{i+1/2,j+1/2}^n-c)-(\phi_{i-1/2,j+1/2}^n-c)}{h_i}=D_x(P_h\phi)_{i,j+1/2}^n,\\
		 \frac{1}{k_j} \int_{y_{j-1/2}}^{y_{j+1/2}} \frac{\partial \phi}{\partial y}(x_{i+1/2},y,t^n) dy = \frac{(\phi_{i+1/2,j+1/2}^n-c)-(\phi_{i+1/2,j-1/2}^n-c)}{k_j}=D_y(P_h\phi)_{{i+1/2},j}^n,
	\end{align*}
	which implies that $(\{W^{x,n}_{i,j+1/2}\},\{W^{y,n}_{i+1/2,j}\},\{(Z+P_h \phi)_{i+1/2,j+1/2}^n\})$ solves 
	\eqref{e7}-\eqref{e9} with $\mathbf{g}$ replaced by $\mathbf{g} + \nabla \phi$. 
	Then \eqref{eq:discrte_invariance} follows immediately from the unique solvability of the MAC scheme (see \cite{LI20181499,Rui2016}).  
	This completes the proof.
\end{proof}
\color{black}

\subsection{Discrete Pressure-Robustness}

Let $\mathbf{\beta} = \mathbf{g} - \nabla p$. Properties \eqref{e24} and \eqref{eq:discrte_invariance} demonstrate that $\mathbf{g}$ and $\mathbf{\beta}$ correspond to the same velocity solution in the continuous and discrete cases.
In addition, in order to solve the problem that the true solution of the pressure does not belong to the discrete pressure space and inspired by the discrete invariance property in Lemma \ref{lem:discrete_invariance}, we construct the new splitting method to decompose the problem \eqref{e1} as
\begin{equation}\label{superposition1}
\left\{
\begin{array}{l}
\displaystyle\frac{\partial \mathbf{u^{(1)}}}{\partial t}-\mu \Delta \mathbf{u^{(1)}}+\nabla p^{(1)}=\mathbf{\beta},~~~~(x,y,t) \in \Omega\times J,\\
\displaystyle\nabla\cdot \mathbf{u^{(1)}}=0,~~~~~~~~~~~~~~~~~~(x,y,t) \in\Omega\times J,\\
\displaystyle\mathbf{u^{(1)}}=0,~~~~~~~~~~~~~~~~~~~~~~(x,y,t) \in \partial \Omega\times J,\\
\mathbf{u^{(1)}}(x,y,0)=0,~~~~~~(x,y) \in \Omega.
\end{array}
\right.
\end{equation}
and
\begin{equation}\label{superposition2}
\left\{
\begin{array}{l}
\displaystyle\frac{\partial \mathbf{u^{(2)}}}{\partial t}-\mu \Delta \mathbf{u^{(2)}}+\nabla p^{(2)}=\nabla p,~~~~(x,y,t) \in \Omega\times J,\\
\displaystyle\nabla\cdot \mathbf{u^{(2)}}=0,~~~~~~~~~~~~~~~~~~(x,y,t) \in\Omega\times J,\\
\displaystyle\mathbf{u^{(2)}}=0,~~~~~~~~~~~~~~~~~~~~~~(x,y,t) \in \partial \Omega\times J,\\
\mathbf{u^{(2)}}(x,y,0)=0,~~~~~~(x,y) \in \Omega.
\end{array}
\right.
\end{equation}
where we can easily get that $\mathbf{u^{(1)}} = \mathbf{u},~p^{(1)} = 0,~\mathbf{u^{(2)}} = 0,~p^{(2)} = p$.

According to the reconstruction above, we can obtain that
\begin{align}
&\frac{1}{h_i} \int_{x_{i-1/2}}^{x_{i+1/2}} \frac{\partial u^{ x}}{\partial t}(x,y_{j+1/2},t^n) dx - \frac{\mu}{h_i} \int_{x_{i-1/2}}^{x_{i+1/2}} \left(\frac{\partial^2 u^{x}}{\partial x^2}(x,y_{j+1/2},t^n) + \frac{\partial^2 u^{x}}{\partial y^2}(x,y_{j+1/2},t^n)\right) dx\notag\\
&~~~~~~~~~~~~~~~~~~~~~~~~~~~~~ = \hat{\beta}^{x,n}_{i,j+1/2}, \ \ 1\leq i\leq N_x-1,0\leq j\leq N_y-1,\label{e21*} \\
&\frac{1}{k_i} \int_{y_{j-1/2}}^{y_{j+1/2}} \frac{\partial u^y}{\partial t}(x_{i+1/2},y,t^n) dy - \frac{\mu}{k_i} \int_{y_{j-1/2}}^{y_{j+1/2}} \left(\frac{\partial^2 u^y}{\partial x^2}(x_{i+1/2},y,t^n) + \frac{\partial^2 u^y}{\partial y^2}(x_{i+1/2},y,t^n)\right)dy\notag\\
&~~~~~~~~~~~~~~~~~~~~~~~~~~~~~ = \hat{\beta}^{y,n}_{i+1/2,j}, \ \ 0\leq i\leq N_x-1,1\leq j\leq N_y-1,\label{e22*} 
\end{align}
where $\displaystyle \hat{\beta}^{n}_{i,j+1/2} = \frac{1}{h_i} \int_{x_{i-1/2}}^{x_{i+1/2}} \beta^{x}(x,y_{j+1/2},t^n) dx,~\hat{\beta}^{y,n}_{i+1/2,j} = \frac{1}{k_j}\int_{y_{j-1/2}}^{y_{j+1/2}} \beta^{y}(x_{i+1/2},y,t^n) dy$.

Hence \eqref{e7}-\eqref{e9} can be divided into  
\begin{align}
&d_{t}{W}^{(1),x,n}_{i,j+1/2}-\mu \frac{d_xW^{(1),x,n}_{i+1/2,j+1/2}-d_xW^{(1),x,n}_{i-1/2,j+1/2}}{h_i}
-\mu \frac{D_yW^{(1),x,n}_{i,j+1}-D_yW^{(1),x,n}_{i,j}}{k_{j+1/2}}\notag\\
&~~~~~+D_xZ^{(1),n}_{i,j+1/2}=\hat{\beta}^{x,n}_{i,j+1/2},\ \ 1\leq i\leq N_x-1,0\leq j\leq N_y-1,\label{dise1}\\
&d_{t}{W}^{(1),y,n}_{i+1/2,j}-\mu \frac{D_xW^{(1),y,n}_{i+1,j}-D_xW^{(1),y,n}_{i,j}}{h_{i+1/2}}
-\mu \frac{d_yW^{(1),y,n}_{i+1/2,j+1/2}-d_yW^{(1),y,n}_{i+1/2,j-1/2}}{k_{j}}\notag\\
&~~~~~+D_yZ^{(1),n}_{i+1/2,j}=\hat{\beta}^{y,n}_{i+1/2,j},\ \ 0\leq i\leq N_x-1,1\leq j\leq N_y-1,\label{dise2}\\
&d_xW^{(1),x,n}_{i+1/2,j+1/2}+d_yW^{(1),y,n}_{i+1/2,j+1/2}\notag\\
&~~~~~=0,\ \ 0\leq i\leq N_x-1,0\leq j\leq N_y-1,\label{dise3} 
\end{align}
and
\begin{align}
&d_{t}{W}^{(2),x,n}_{i,j+1/2}-\mu \frac{d_xW^{(2),x,n}_{i+1/2,j+1/2}-d_xW^{(2),x,n}_{i-1/2,j+1/2}}{h_i}
-\mu \frac{D_yW^{(2),x,n}_{i,j+1}-D_yW^{(2),x,n}_{i,j}}{k_{j+1/2}}\notag\\
&~~~~~+D_xZ^{(2),n}_{i,j+1/2}=D_xp^{n}_{i,j+1/2},\ \ 1\leq i\leq N_x-1,0\leq j\leq N_y-1,\label{dise4}\\
&d_{t}{W}^{(2),y,n}_{i+1/2,j}-\mu \frac{D_xW^{(2),y,n}_{i+1,j}-D_xW^{(2),y,n}_{i,j}}{h_{i+1/2}}
-\mu \frac{d_yW^{(2),y,n}_{i+1/2,j+1/2}-d_yW^{(2),y,n}_{i+1/2,j-1/2}}{k_{j}}\notag\\
&~~~~~+D_yZ^{(2),n}_{i+1/2,j}=D_yp^n_{i+1/2,j},\ \ 0\leq i\leq N_x-1,1\leq j\leq N_y-1,\label{dise5}\\
&d_xW^{(2),x,n}_{i+1/2,j+1/2}+d_yW^{(2),y,n}_{i+1/2,j+1/2}\notag\\
&~~~~~=0.\ \ 0\leq i\leq N_x-1,0\leq j\leq N_y-1.\label{dise6} 
\end{align} 
Recalling \eqref{eq:discrte_invariance},  we have $\mathbf{W}^{(1),n}=\mathbf{W}^n$ and $\mathbf{W}^{(2),n}=0$ for all $n\in\mathbb{N}$.

Therefore, for velocity estimate we only need to prove
\begin{equation}
	\|\mathbf{W}^{(1),n}-\mathbf{u}^{(1),n}\| \leq C(h^2 +k^2 + \Delta t),
\end{equation}
where the positive constant $C$ is independent of $h,~k,~\Delta t$,   pressure $p$ and undisplayed dependent on $1/\mu$.

\subsection{Stability}
Motivated by duality arguments, we establish the following lemma for pressure error bounds, which formally parallels the discrete LBB condition \cite{Rui2016}. The derived lemma can be regarded as a framework with extendable capabilities to address limitations inherent in traditional LBB-based analyses, which will be left as subjects of future endeavor.

We first define the operator $\mathbf{I}_h: \mathbf{V} \to \mathbf{V}_h$, such that
\begin{equation*}
\mathbf{I}_h \mathbf{v} = (I^x_h v^x, I^y_h v^y),~\mathbf{v} = (v^x,v^y) \in \mathbf{V},
\end{equation*}
where
\begin{equation}
\aligned
(I^x_h v^x)_{i, j+1/2} = \frac{1}{k_{j+1/2}} \int_{y_j}^{y_{j+1}}v^x(x_i,y)dy, \ \ 
(I^y_h v^y)_{i+1/2, j} = \frac{1}{h_{j+1/2}} \int_{x_i}^{x_{i+1}}v^y(x,y_j)dx.
\endaligned
\end{equation}

%
%
%
\begin{lem}\label{dual-problem}
	For each $q \in W_h$, there exists $\mathbf{v} \in \mathbf{V}_h$ such that 
	\begin{equation}\label{div*}
	(d_x v^x + d_y v^y)_{i+1/2,j+1/2} = q_{i+1/2,j+1/2},~\|\mathbf{v}\|_{l^2}+\|D\mathbf{v}\|_{l^2} \leq C\|q\|_{l^2},
	\end{equation}
	where the positive constant $C$ is independent of $h$ and $k$. 
\end{lem}
\begin{proof}
	This is a standard duality argument for $q \in W_h$. Firstly, since the divergence operator is an isomorphism
	from $\mathbf{V}^\bot = \{\mathbf{v} \in H^1_0 |~\mathbf{u} \in H^1_0,~\nabla \cdot \mathbf{u} = 0,~(\nabla \mathbf{u},\nabla \mathbf{v}) = 0\}$ onto $\mathbf{W}$(cf. \cite{girault1986finite}), there exists $\mathbf{w} \in H^1_0$ such that
	\begin{equation}\label{dualeq}
	\aligned
	\nabla \cdot \mathbf{w} = q, \ \forall q \in W_h, \\
	\|\mathbf{w}\|_{H^1} \leq C_1\|q\|_{L^2},
	\endaligned
	\end{equation}
	Setting $\mathbf{v} = \mathbf{I}_h \mathbf{w} \in \mathbf{V}_h$, we find that
	\begin{equation}
	d_x v^x_{i+1/2,j+1/2} + d_y v^y_{i+1/2,j+1/2} = q_{i+1/2,j+1/2}.
	\end{equation}
	Hence we have \cite{gallouet2012w1} 
	\begin{equation}
	\|D \mathbf{v}\|_{l^2} \leq C_2 \|\mathbf{w}\|_{H^1},
	~~\|\mathbf{v}\|_{l^2} \leq C_3 \|D \mathbf{v}\|_{l^2},
	\end{equation}
	which derives that
	\begin{equation}
	\|\mathbf{v}\|_{l^2} + \|D \mathbf{v}\|_{l^2} \leq C\|q\|_{l^2},
	\end{equation}
	where $C = C_1C_2(C_3 + 1)$.
\end{proof}

It is easy to verify that the following discrete integration-by-part formula holds.
\medskip
\begin{lem}\label{le2}
	\cite{Rui2013Block} Let $q^{x,n}_{i+1/2,j+1/2},~q^{y,n}_{i+1/2,j+1/2},~V_{i,j+1/2}^{x,n}~and~ V_{i+1/2,j}^{y,n} $ be any values such that $V_{0,j+1/2}^{x,n}=V_{N_x,j+1/2}^{x,n}=V_{i+1/2,0}^{y,n}=V_{i+1/2,N_y}^{y,n}=0$, then
	$$(D_xq^{x,n},V^{x,n})_{l^2,T,M}=-(q^{x,n},d_xV^{x,n})_{l^2,M},$$
	$$(D_yq^{y,n},V^{y,n})_{l^2,M,T}=-(q^{y,n},d_yV^{y,n})_{l^2,M}.$$
\end{lem}

\begin{thm}\label{stability}
	(Stability)
	For the RMAC scheme (\ref{e7})-(\ref{e9}) and supposing that $\Delta t$ is sufficient small, there exists a constant $C$ independent of the viscosity such that:
	\begin{align}
	\|\mathbf{W}^m\|^2_{l^2} &\leq C\sum_{n = 1}^{m} \Delta t \|\hat{\beta}^n\|^2_{l^2}, \ 1 \leq m \leq N,  \label{Sta1}\\
	\|Z^m\|^2_{l^2,M} &\leq C (\sum_{n = 2}^{m} \Delta t \|d_t\mathbf{f}^n\|^2_{l^2} + \mu\sum_{n = 1}^{m} \Delta t \|\mathbf{f}^n\|^2_{l^2} +\|\mathbf{f}^1\|^2_{l^2} +\|\mathbf{f}^m\|^2_{l^2}), 1 \leq m \leq N \label{Sta2}.
	\end{align}
\end{thm}
\begin{proof}
	Similarly to the stability results in \cite{LI20181499} and recalling (\ref{dise1})-(\ref{dise3}), we can easily obtain that
	$$\|\mathbf{W}^{(1),m}\|\leq C\sum_{n = 1}^{m} \Delta t \|\hat{\beta}^n\|^2_{l^2}.$$ 
	 Recalling the fact $\mathbf{W}^{m}=\mathbf{W}^{(1),m}$, we derive the desired result (\ref{Sta1}). In addition,  using the stability results in \cite{LI20181499} and noting (\ref{e7})-(\ref{e9}) yields
	\begin{equation}\label{sta}
	\aligned
	\mu \|D\mathbf{W}^m\|^2_{l^2}&\leq C \sum_{n = 1}^{m} \Delta t \|\mathbf{f}^n\|^2_{l^2}.
	\endaligned
	\end{equation}
	
	For the discrete function $\mathbf{v}^n \in \mathbf{V}_h$, (\ref{e7}) times $v^{x,n}_{i,j+1/2} h_ik_{j+1/2}$ and make summation on $i,~j$ for $i = 1, \dots ,N_x - 1,~j = 0, \dots , N_y - 1$. (\ref{e8}) times $v^{y,n}_{i+1/2,j} h_{i+1/2}k_j$ and make summation on $i,~j$ for $i = 0, \dots ,N_x - 1,~j = 1, \dots , N_y - 1$. Adding them results in
	\begin{equation}\label{d_t}
	\aligned
	&(d_t W^{x,n}, v^{x,n})_{l^2,T,M} + (d_t W^{y,n}, v^{y,n})_{l^2,M,T} + \mu ((d_x W^{x,n},d_x v^{x,n})_{l^2,M} \\
	& ~+ (D_y W^{x,n},D_y v^{x,n})_{l^2,T_y} + (D_x W^{y,n},D_x v^{y,n})_{l^2,T_x} \\
	&~+ (d_y W^{y,n}, d_y v^{y,n})_{l^2,T_y}) - (Z^n, d_xv^{x,n} + d_y v^{y,n})_{l^2,M}\\
	=&(f^{x,n},v^{x,n})_{l^2,T,M} + (f^{y,n},v^{y,n})_{l^2,M,T}.
	\endaligned
	\end{equation}
	
	By using Lemmas \ref{dual-problem} and \ref{le2} and noting (\ref{d_t}), we obtain that there exists $\mathbf{v}^n \in \mathbf{V}_h$ such that
	\begin{equation}
	\aligned
	\|Z^n\|^2_{l^2,M} &= (Z^n,d_x v^{x,n} + d_y v^{y,n})_{l^2,M}\\
	&= (d_t W^{x,n} - f^{x,n}, v^{x,n})_{l^2,T,M} + (d_t W^{y,n} - f^{y,n}, v^{y,n})_{l^2,M,T} + \mu ((d_x W^{x,n},d_x v^{x,n})_{l^2,M} \\
	&~~ + (D_y W^{x,n},D_y v^{x,n})_{l^2,T_y} + (D_x W^{y,n},D_x v^{y,n})_{l^2,T_x} + (d_y W^{y,n}, d_y v^{y,n})_{l^2,T_y})\\
	&\leq \frac{C}{2} (\|d_t \mathbf{W}^n\|^2_{l^2} + \mu^2\|D\mathbf{W}^n\|^2_{l^2} + \|\mathbf{f}^n\|^2_{l^2}) + \frac{1}{2}\|Z^n\|^2_{l^2,M},
	\endaligned
	\end{equation}
	and then
	\begin{equation}\label{ep}
	\|Z^n\|^2_{l^2,M} \leq C (\|d_t \mathbf{W}^n\|^2_{l^2} + \mu^2 \|D\mathbf{W}^n\|^2_{l^2} + \|\mathbf{f}^n\|^2_{l^2}),  
	\end{equation}
	where $C$ is a constant independent of viscosity.
	
	Next we give the stability result for $\|d_t \mathbf{W}^n\|^2_{l^2}$.  By making the temporal difference quotient of (\ref{e7}) and (\ref{e8}) and using a similar procedure to \eqref{d_t}, we have 
	\begin{equation}
	\aligned
	\|d_t\mathbf{W}^m\|^2_{l^2} + 2\mu \sum_{n = 2}^{m} \Delta t \|D d_t \mathbf{W}^n\|^2_{l^2}&\leq   \|d_t\mathbf{W}^1\|^2_{l^2} + 2\sum_{n = 2}^{m} \Delta t \|d_t\mathbf{f}^n\|^2_{l^2} + \frac{1}{2}\sum_{n = 2}^{m} \Delta t \|d_t\mathbf{W}^n\|^2_{l^2}.
	\endaligned
	\end{equation}
	Using Grownwall's inequality yields
	\begin{equation}\label{sta*}
	\|d_t\mathbf{W}^m\|^2_{l^2} \leq C(\|d_t\mathbf{W}^1\|^2_{l^2} + \sum_{n = 2}^{m} \Delta t \|d_t\mathbf{f}^n\|^2_{l^2}).
	\end{equation}
	As for $\|d_t\mathbf{W}^1\|^2_{l^2}$, we set $\mathbf{v}^n = d_t\mathbf{W}^n$ in (\ref{d_t}) for $n = 1$, we will get that 
	\begin{equation}\label{sta**}
	\aligned
	\|d_t \mathbf{W}^1\|^2_{l^2} + \mu \frac{\|D \mathbf{W}^1\|^2}{\Delta t} \leq \frac{1}{2} \|\mathbf{f}^1\|^2_{l^2} + \frac{1}{2}\|d_t \mathbf{W}^1\|^2_{l^2}.
	\endaligned
	\end{equation}
	Hence we can easily obtain (\ref{Sta2}) by combining (\ref{sta}) (\ref{ep}) (\ref{sta*}) and (\ref{sta**}).
\end{proof}
\section{Error Analysis for Discrete RMAC Scheme}
In this section,  by introducing the new auxiliary function depending on the velocity and discretizing parameters and establishing several lemmas, we establish robust superconvergence results for both velocity and pressure on non-uniform grids under the assumption of sufficient regularity for the exact solutions $\mathbf{u}$ and $p$.  To simplify notation, we adopt the convention $\mathbf{u} = \mathbf{u}^{(1)}$ throughout this section.

\subsection{Consistency Analysis}

Inspired by the established superconvergence analysis in \cite{li2018superconvergence,Rui2016}, we construct the new auxiliary function to derive the superconvergence results on non-uniform grids.

Define the discrete interpolations of $u^{x,n}$ and $u^{y,n}$ in $\Omega$ and on the boundary of $\Omega$ as
follows. Define for $i=0,1,\cdots,N_x,~j=0,1,\cdots,N_y-1$

\begin{equation}\label{e25}
\widetilde{u}^{x,n}_{i,j+1/2}=
\begin{array}{l}
\displaystyle
\frac{1}{k_{j+1/2}}\int_{y_{j}}^{y_{j+1}}u^{x,n}(x_i,y)dy-\frac{k_{j+1/2}^2}{6}\frac{\partial^2u^{x,n}_{i,j+1/2}}{\partial y^2},~~~~n\geq 0.
\end{array}
\end{equation}

And for $i=0,1,\cdots,N_x-1,~j=0,1,\cdots,N_y$
\begin{equation}\label{e26}
\widetilde{u}^{y,n}_{i+1/2,j}=
\begin{array}{l}
\displaystyle
\frac{1}{h_{i+1/2}}\int_{x_{i}}^{x_{i+1}}u^{y,n}(x,y_j)dx-\frac{h_{i+1/2}^2}{6}\frac{\partial^2u^{y,n}_{i+1/2,j}}{\partial y^2},~~~~n\geq 0.
\end{array}
\end{equation} 

Extend the value of $\widetilde{u}^{x,n}$ and $\widetilde{u}^{y,n}$ in the following ways,
\begin{equation}\label{e27}
\aligned
\widetilde{u}^{x,n}_{i,-1/2}=\widetilde{u}^{x,n}_{i,0}=u^{x,n}_{i,0},
~\widetilde{u}^{x,n}_{i,N_y+1/2}=\widetilde{u}^{x,n}_{i,N_y}=u^{x,n}_{i,N_y},~~~i=1,2,\cdots,N_x-1,
\endaligned
\end{equation}

\begin{equation}\label{e28}
\aligned
\widetilde{u}^{y,n}_{-1/2,j}=\widetilde{u}^{y,n}_{0,j}=u^{y,n}_{0,j},
~\widetilde{u}^{y,n}_{N_x+1/2,j}=\widetilde{u}^{y,n}_{N_x,j}=u^{y,n}_{N_x,j},  ~~~j=1,2,\cdots,N_y-1.
\endaligned
\end{equation}

First, we give the main results of consistency analysis for \eqref{e21*} and \eqref{e22*}.
\medskip
\begin{lem}(Consistency)\label{lemcon}
	Let $\mathbf{u}=(u^x,u^y)$ be the solution of (\ref{e20}). Let $\widetilde{u}^{x,n},~\widetilde{u}^{y,n}$ be defined by (\ref{e25}) and (\ref{e26}). Then (\ref{e21*}) and (\ref{e22*}) can be transformed into the following form:
	\begin{equation}\label{e39}
	\aligned
	&d_{t}{\widetilde{u}}^{x,n}_{i,j+1/2}-\mu \frac{d_x\widetilde{u}^{x,n}_{i+1/2,j+1/2}-d_x\widetilde{u}^{x,n}_{i-1/2,j+1/2}}{h_i}
	-\mu \frac{D_y\widetilde{u}^{x,n}_{i,j+1}
		-D_y\widetilde{u}^{x,n}_{i,j}}{k_{j+1/2}}=\hat{\beta}_{i,j+1/2}^{x,n}+E^{x,n}_{i,j+1/2}\\
	&~~~~~~-\frac{\gamma^{x,n}_{i+1/2,j+1/2}-\gamma^{x,n}_{i-1/2,j+1/2}}{h_i}-\mu \frac{\epsilon^{x,n}_{i+1/2,j+1/2}-\epsilon^{x,n}_{i-1/2,j+1/2}}{h_i}
	-\mu \frac{\epsilon^{x,y,n}_{i,j+1}-\epsilon^{x,y,n}_{i,j}}{k_{j+1/2}},
	\endaligned
	\end{equation}
    \begin{equation}\label{e39*}
    \aligned
	&d_{t}{\widetilde{u}}^{y,n}_{i+1/2,j}-\mu \frac{D_x\widetilde{u}^{y,n}_{i+1,j}-D_x\widetilde{u}^{y,n}_{i,j}}{h_{i+1/2}}
	-\mu \frac{d_y\widetilde{u}^{y,n}_{i+1/2,j+1/2}
		-d_y\widetilde{u}^{y,n}_{i+1/2,j-1/2}}{k_j}=\hat{\beta}_{i+1/2,j}^{y,n}+E^{y,n}_{i+1/2,j}\\
	&~~~~~~-\frac{\gamma^{y,n}_{i+1/2,j+1/2}-\gamma^{y,n}_{i+1/2,j-1/2}}{k_j}-\mu \frac{\epsilon^{y,n}_{i+1/2,j+1/2}-\epsilon^{y,n}_{i+1/2,j-1/2}}{k_j}
	-\mu \frac{\epsilon^{y,x,n}_{i+1,j}-\epsilon^{y,x,n}_{i,j}}{h_{i+1/2}},
	\endaligned
	\end{equation}
	where
	\begin{equation*}
	\aligned
	\displaystyle \gamma^{x,n}_{i+1/2,j+1/2} = \frac{h^2_{i+1/2}}{8}\frac{\partial^2 u^{x,n}_{i+1/2,j+1/2}}{\partial t \partial x},&~~\gamma^{y,n}_{i+1/2,j+1/2} = \frac{k^2_{j+1/2}}{8}\frac{\partial^2 u^{y,n}_{i+1/2,j+1/2}}{\partial t \partial y}\\
	 \epsilon^{y,n}_{i+1/2,j+1/2},~\epsilon^{x,n}_{i+1/2,j+1/2},~\epsilon^{y,x,n}_{i,j},~\epsilon^{x,y,n}_{i,j} &= O(\|\mathbf{u}\|_{L^{\infty}(J;W^{3,\infty}(\Omega))} (h^2+k^2)),\\
	 \displaystyle E^{x,n}_{i,j+1/2},~E^{y,n}_{i+1/2,j} &= (1 + \mu)O(\|\mathbf{u}\|_{W^{2,\infty}(J;W^{4,\infty}(\Omega))} (h^2+k^2+\Delta t)).
	\endaligned
	\end{equation*}
	For brevity, the proof details for the above lemma will be specified separately in the Appendix A.
\end{lem}

  Next we give the following lemma to point out the essential difference between the divergence free operators in continuous and discrete cases. 
  \medskip
\begin{lem}\label{le5}
	 Let $\mathbf{u}=(u^x,u^y)$ be the solution of (\ref{e20}). Let $\widetilde{u}^{x,n},~\widetilde{u}^{y,n}$ be defined by (\ref{e25}) and (\ref{e26}). Then for $0\leq i\leq N_x-1,~0\leq j\leq N_y-1$ there holds
	\begin{equation}\label{e29}
	\aligned
	\quad [d_x\widetilde{u}^{x}+d_y\widetilde{u}^{y}]_{i+1/2,j+1/2}^n&=O(\|\mathbf{u}\|_{L^{\infty}(J;W^{3,\infty}(\Omega))}(h^2+k^2)),\\
	d_t[d_x\widetilde{u}^{x}+d_y\widetilde{u}^{y}]_{i+1/2,j+1/2}^n&=O(\|\mathbf{u}\|_{W^{2,\infty}(J;W^{3,\infty}(\Omega))}(h^2+k^2 + \Delta t)),\\
	d_td_t[d_x\widetilde{u}^{x}+d_y\widetilde{u}^{y}]_{i+1/2,j+1/2}^n&=O(\|\mathbf{u}\|_{W^{3,\infty}(J;W^{3,\infty}(\Omega))}(h^2+k^2+\Delta t)).
	\endaligned
	\end{equation}
\end{lem}
\begin{proof}
	Set
	\begin{equation}
	\aligned
		\hat{u}^{x,n}_{i,j+1/2} = \frac{1}{k_{j+1/2}}\int_{y_j}^{y_{j+1}}u^{x,n}(x_i,y)dy \\
		\hat{u}^{y,n}_{i+1/2,j} = \frac{1}{h_{i+1/2}}\int_{x_i}^{x_{i+1}}u^{y,n}(x,y_j)dx,
	\endaligned
	\end{equation}
	Recalling the divergence-free condition, we can obtain that 
	\begin{equation}
		\aligned
		d_x\hat{u}_{i+1/2,j+1/2} + d_y\hat{u}_{i+1/2,j+1/2} = 0.
		\endaligned
	\end{equation}
	Then (\ref{e29}) becomes a standard exercise.
\end{proof}

Next to obtain the $\mu$-robustness on non-uniform grids for the constructed RMAC scheme, we establish the follow lemma. Our analysis reveals a critical deficiency in standard energy methods in \cite{Rui2016,li2018superconvergence}. Specifically in detail, the temporal truncation error component $\gamma^n$ is not multiplied by $\mu$, thereby fundamentally compromising the $\mu$-robustness. To address this limitation, we exploit the intrinsic second-order accuracy of mixed derivatives $D_y \gamma^{x,n},~D_x \gamma^{y,n}$ on  non-uniform grids, which is an essential improvement compared to the result of not preserving the $\mu$-robustness in \cite{Rui2016,li2018superconvergence}.
\medskip
\begin{lem}\label{cross}
	Let $\gamma = (\gamma^x,\gamma^y)$ be definded in Lemma \ref{lemcon}. For any $\mathbf{v} \in \mathbf{V}_h$ satisfies \begin{equation}\label{ediv}
		d_x v^x + d_y v^y = O(h^2 + k^2 + \Delta t),
	\end{equation}	
	there holds
	\begin{equation}
		(D_x \gamma^{x,n},v^x)_{l^2,T,M} + (D_y \gamma^{y,n},v^y)_{l^2,M,T} \leq \epsilon\|\mathbf{v}\|^2_{l^2} + O(h^2 + k^2 + \Delta t)^2, 
	\end{equation}
	where the positive constant $\epsilon$ can be arbitrarily small.
\end{lem}

\begin{proof}
	By using Lemma \ref{le2} and Cauchy-Schwarz inequality, we can obtain that
	\begin{equation}
	\aligned
		(D_x \gamma^{x,n}&,v^x)_{l^2,T,M} + (D_y \gamma^{y,n},v^y)_{l^2,M,T} \\
		&= -(\gamma^{x,n},d_xv^x)_{l^2,M} -(\gamma^{y,n},d_yv^y)_{l^2,M} \\
		&= -(\gamma^{x,n}+\gamma^{y,n},d_xv^x+d_yv^y)_{l^2,M} + (\gamma^{x,n},d_yv^x)_{l^2,M}+(\gamma^{y,n},d_xv^x)_{l^2,M}\\
		&= -(\gamma^{x,n}+\gamma^{y,n},d_xv^x+d_yv^y)_{l^2,M} -(D_y \gamma^{x,n},v^y)_{l^2,M,T} - (D_x \gamma^{y,n},v^x)_{l^2,T,M}\\
		&\leq \epsilon\|\mathbf{v}\|^2_{l^2}+\frac{1}{4\epsilon}(\|D_y\gamma^{x,n}\|^2_{l^2,M,T}+\|D_x\gamma^{y,n}\|^2_{l^2,T,M})-(\gamma^{x,n}+\gamma^{y,n},d_xv^x+d_yv^y)_{l^2,M}.
	\endaligned
	\end{equation}
	Recalling (\ref{ediv}) and the defination of $\gamma$, we obtain the desired result.
\end{proof}

For simplicity, we set 
\begin{equation}\label{e35}
\aligned
e^{x,n}_{i,j+1/2}=(W^x-\widetilde{u}^x)_{i,j+1/2}^n,~e^{y,n}_{i+1/2,j}=(W^y-\widetilde{u}^y)_{i+1/2,j}^n,~e^{p,n}_{i+1/2,j+1/2}=(Z^{(1)}-0)_{i+1/2,j+1/2}^n.
\endaligned
\end{equation}
From (\ref{e6}), we have
\begin{equation}\label{e36}
\left\{
\begin{array}{l}
e_{0,j+1/2}^{x,n}=e_{N_x,j+1/2}^{x,n}=0,\quad 0\leq j\leq N_y-1,\\
e_{i,0}^{x,n}=e_{i,N_y}^{x,n}=0,\quad 0\leq i\leq N_x,\\
e_{0,j}^{y,n}=e_{N_x,j}^{y,n}=0, \quad 0\leq j\leq N_y,\\
e_{i+1/2,0}^{y,n}=e_{i+1/2,N_y}^{y,n}=0,\quad 0\leq i\leq N_x-1.
\end{array}
\right.
\end{equation} 
\subsection{Error Estimates for the Velocity}
While both RMAC and MAC schemes maintain exact discrete incompressibility for numerical velocity solutions, their auxiliary velocity $\tilde{\mathbf{u}}$ inherently violate this constraint, which fundamentally compromises the discrete divergence-free constraint on the velocity error $\mathbf{e}^n$.  Hence the main contribution in this section is to establish robust superconvergence for the constructed RMAC scheme rigorously, as formally established in the following theorem.
\medskip
\begin{thm}\label{th9}
	Suppose that the analytical solutions $\textbf{u}$ and $p$ are sufficiently smooth. For the RMAC scheme  (\ref{e7})-(\ref{e9}), there exists a positive constant $C$ independent of $h,~k,~\Delta t,$ pressure $p$ and undisplayed dependent on $1/\mu$, such that for $1 \leq m \leq N$
	\begin{align} \label{e_nonuniform_erroru}
	\|\mathbf{W}^{m}-\mathbf{u}^{m}\|_{l^2}&\leq C(\mathbf{u},\partial_t\mathbf{u},\partial_{tt}\mathbf{u},T)(h^2+k^2+\Delta t).
	\end{align}
	Furthermore, we can obatin the stronger error estimate with $\Delta t = O(h^2)$ on uniform grids,
	\begin{align} \label{e_uniform_erroru}
		\|\mathbf{W}^{m}-\mathbf{u}^{m}\|_{l^\infty}&\leq C(\mathbf{u},\partial_t\mathbf{u},\partial_{tt}\mathbf{u},T)h^2.
	\end{align}

\end{thm}

\begin{proof}
   Subtracting (\ref{e39}) from (\ref{dise1}) and recalling that $\mathbf{W}^{(1),n}=\mathbf{W}^n$, we can obtain
   \begin{equation}\label{e40}
   \aligned
   &d_{t}{e}^{x,n}_{i,j+1/2}-\mu \frac{d_xe^{x,n}_{i+1/2,j+1/2}-d_xe^{x,n}_{i-1/2,j+1/2}}{h_i}
   -\mu \frac{D_ye^{x,n}_{i,j+1}-D_ye^{x,n}_{i,j}}{k_{j+1/2}}+D_xe^{p,n}_{i,j+1/2}\\
   =&E^{x,n}_{i,j+1/2}+ \frac{\gamma^{x,n}_{i+1/2,j+1/2}-\gamma^{x,n}_{i-1/2,j+1/2}}{h_i}+\mu \frac{\epsilon^{x,n}_{i+1/2,j+1/2}-\epsilon^{x,n}_{i-1/2,j+1/2}}{h_i}
   +\mu \frac{\epsilon^{x,y,n}_{i,j+1}-\epsilon^{x,y,n}_{i,j}}{k_{j+1/2}}.
   \endaligned
   \end{equation}
   
   For a discrete function $\{v^{x}_{i,j+1/2}\}$ such that $v^{x}_{i,j+1/2}|_{\partial \Omega}=0$,  let us multiply  (\ref{e40}) by $h_ik_{j+1/2}v^{x}_{i,j+1/2}$
   and sum for $i,j$ with $i=1,\cdots,N_x-1,~j=0,\cdots,N_y-1$. By Lemma \ref{le2}, we can have that
   \begin{equation}\label{e41}
   \aligned
   &(d_te^{x,n},v^{x})_{l^2,T,M}+\mu (d_xe^{x,n},d_xv^{x})_{l^2,M}+\mu (D_ye^{x,n},D_yv^{x})_{l^2,T_y}-(e^{p,n},d_xv^{x})_{l^2,M}\\
   =&(D_x\gamma^{x,n},v^{x})_{l^2,T,M}-\mu(\epsilon^{x,n},d_xv^{x})_{l^2,M}-\mu (\epsilon^{x,y,n},D_yv^{x})_{l^2,T_y}
   +(E^{x,n},v^{x})_{l^2,T,M},
   \endaligned
   \end{equation}
   Similarly in the $y$-direction, we have that for discrete function $\{v^{y}_{i+1/2,j}\}$ such that $v^{y}_{i+1/2,j}|_{\partial \Omega}=0$
   \begin{equation}\label{e42}
   \aligned
   &(d_te^{y,n},v^{y})_{l^2,M,T}+\mu (d_ye^{y,n},d_yv^{y})_{l^2,M}+\mu (D_xe^{y,n},D_xv^{y})_{l^2,T_x}-(e^{p,n},d_yv^{y})_{l^2,M}\\
   =&(D_y\gamma^{y,n},v^{y})_{l^2,M,T}-\mu(\epsilon^{y,n},d_yv^{y})_{l^2,M}-\mu (\epsilon^{y,x,n},D_xv^{y})_{l^2,T_x}
   +(E^{y,n},v^{y})_{l^2,M,T},
   \endaligned
   \end{equation}
   Adding (\ref{e41}) and (\ref{e42}) yields
   \begin{equation}\label{e43}
   \aligned
   &(d_te^{x,n},v^{x})_{l^2,T,M}+(d_te^{y,n},v^{y})_{l^2,M,T}
   +\mu (d_xe^{x,n},d_xv^{x})_{l^2,M}\\
   &+\mu (D_ye^{x,n},D_yv^{x})_{l^2,T_y}
   +\mu (d_ye^{y,n},d_yv^{y})_{l^2,M}+\mu (D_xe^{y,n},D_xv^{y})_{l^2,T_x}\\
   &-(e^{p,n},d_xv^{x}+d_yv^{y})_{l^2,M}\\
   =&(D_x\gamma^{x,n},v^{x})_{l^2,T,M} + (D_y\gamma^{y,n},v^{y})_{l^2,M,T}-\mu(\epsilon^{x,n},d_xv^{x})_{l^2,M}\\
   &-\mu (\epsilon^{x,y,n},D_yv^{x})_{l^2,T_y}
   -\mu(\epsilon^{y,n},d_yv^{y})_{l^2,M}-\mu (\epsilon^{y,x,n},D_xv^{y})_{l^2,T_x}\\
   &+(E^{x,n},v^{x})_{l^2,T,M}+(E^{y,n},v^{y})_{l^2,M,T}.
   \endaligned
   \end{equation}
Recalling Lemma \ref{dual-problem}, we have
   \begin{equation}\label{forep}
   \aligned
   & \|e^{p,n}\|^2_{l^2,M} = (e^{p,n},d_x v^{x,n} + d_y v^{y,n})_{l^2,M}\\
   &=(d_te^{x,n},v^{x})_{l^2,T,M}+(d_te^{y,n},v^{y})_{l^2,M,T}
   +\mu (d_xe^{x,n},d_xv^{x})_{l^2,M} +\mu (D_ye^{x,n},D_yv^{x})_{l^2,T_y}
   +\mu (d_ye^{y,n},d_yv^{y})_{l^2,M} \\
   &~~+\mu (D_xe^{y,n},D_xv^{y})_{l^2,T_x}+\mu(\epsilon^{x,n},d_xv^{x})_{l^2,M}+\mu (\epsilon^{x,y,n},D_yv^{x})_{l^2,T_y}
   +\mu(\epsilon^{y,n},d_yv^{y})_{l^2,M}\\
   &~~+\mu (\epsilon^{y,x,n},D_xv^{y})_{l^2,T_x}+(\gamma^{x,n},d_x v^x)_{l^2,M}+(\gamma^{y,n},d_y v^y)_{l^2,M} - (E^{x,n},v^{x})_{l^2,T,M}-(E^{y,n},v^{y})_{l^2,M,T}\\
   &\leq \frac{C}{2} (\|d_t \mathbf{e}^n\|^2_{l^2} + \mu^2\|D\mathbf{e}^n\|^2_{l^2}) + \frac{1}{2}\|e^{p,n}\|^2_{l^2,M}+ (1+\mu+\mu^2)O(\|\mathbf{u}\|_{W^{2,\infty}(J;W^{4,\infty}(\Omega))} (h^2+k^2+\Delta t))^2,
   \endaligned
   \end{equation}
   where $\mathbf{v}^n \in \mathbf{V}_h$.
   Noting that
   $$(q^n,d_tq^n) = \frac{1}{2}d_t\|q^n\|^2 + \frac{\Delta t}{2}\|d_tq^n\|^2,$$
   and setting $\mathbf{v} = d_t\mathbf{e}^n$ in (\ref{e43}), and summing for $1\leq n \leq m \leq N$, we have that  
   \begin{equation}\label{e43+}
   \aligned
   &2\sum_{n = 1}^{m}\Delta t \|d_t\mathbf{e}^n\|^2_{l^2} +\mu\|D\mathbf{e}^m\|^2_{l^2}-\mu\|D\mathbf{e}^0\|^2_{l^2}\\
   \leq& 2\sum_{n = 1}^{m}\Delta t (e^{p,n},d_t(d_x e^{x,n} + d_y e^{y,n}))_{l^2,M} + \mu\sum_{n=2}^{m}\Delta t \left((d_t\epsilon^{x,n},d_xe^{x,n-1})_{l^2,M} \right.\\
   &\left.+(d_t\epsilon^{x,y,n},D_ye^{x,n-1})_{l^2,T_y}
   +(d_t\epsilon^{y,n},d_ye^{y,n-1})_{l^2,M}
   +(d_t\epsilon^{y,x,n},D_xe^{y,n-1})_{l^2,T_x}\right)\\
   &- \mu((\epsilon^{x,m},d_xe^{x,m})_{l^2,M} - (\epsilon^{x,y,m},D_ye^{x,m})_{l^2,T_y} - (\epsilon^{y,m},d_ye^{y,m})_{l^2,M}- (\epsilon^{y,x,m},D_xe^{y,m})_{l^2,T_x} )\\
   & + \mu ((\epsilon^{x,1},d_xe^{x,0})_{l^2,M} + (\epsilon^{x,y,1},D_ye^{x,0})_{l^2,T_y} + (\epsilon^{y,1},d_ye^{y,0})_{l^2,M} + (\epsilon^{y,x,1},D_xe^{y,0})_{l^2,T_x})\\
   &+ 2 \sum_{n=1}^{m} \Delta t \left((E^{x,n},d_te^{x,n})_{l^2,T,M}+(E^{y,n},d_te^{y,n})_{l^2,M,T}+(D_x\gamma^{x,n},d_te^{x,n})_{l^2,T,M}+(D_y\gamma^{y,n},d_te^{y,n})_{l^2,M,T}\right) 
   \endaligned
   \end{equation}
   By using Lemma \ref{cross}, (\ref{forep}) and noting that $\mathbf{e}^0 = 0$, we can get that
   \begin{equation}
   \aligned
   2\sum_{n = 1}^{m}\Delta t \|d_t\mathbf{e}^n\|^2_{l^2} + \mu \|D\mathbf{e}^m\|^2_{l^2} \leq& \sum_{n=1}^{m} \Delta t \|d_t\mathbf{e}^n\|^2_{l^2} +  \frac{\mu+\mu^2}{2}\sum_{n=0}^{m}\Delta t \|D\mathbf{e}^n\|^2_{l^2} + \frac{\mu}{2}\|D\mathbf{e}^m\|^2_{l^2} \\
   &+ (1+\mu+\mu^2)O(\|\mathbf{u}\|_{W^{2,\infty}(J;W^{4,\infty}(\Omega))} (h^2+k^2+\Delta t))^2.
   \endaligned
   \end{equation}
   Applying Gronwall's inequality, we conclude that
   \begin{equation}\label{4.24}
   \aligned
   &\sum_{n = 1}^{m}\Delta t \|d_t\mathbf{e}^n\|^2_{l^2} + \mu \|D\mathbf{e}^m\|^2_{l^2} \leq e^{\mu T}(1+\mu+\mu^2)O(\|\mathbf{u}\|_{W^{2,\infty}(J;W^{4,\infty}(\Omega))} (h^2+k^2+\Delta t))^2.
   \endaligned
   \end{equation}
   
   Next setting $\mathbf{v}=\mathbf{e}^n$ in (\ref{e43}) and summing for $n$ from $1$ to $m$($m \le N$), we have that
   \begin{equation}
   \aligned
   &\|\mathbf{e}^m\|_{l^2}^2- \|\mathbf{e}^0\|_{l^2}^2+2\mu \sum_{n = 1}^{m} \Delta t \|D\mathbf{e}^n\|^2_{l^2} \\
   \leq &\sum_{n=1}^{m}2 \Delta t((e^{p,n},d_xe^{x,n}+d_ye^{y,n})_{l^2,M}-\mu(\epsilon^{x,n},d_xe^{x,n})_{l^2,M}-\mu(\epsilon^{y,n},d_ye^{y,n})_{l^2,M}\\
   &- \mu(\epsilon^{x,y,n},D_ye^{x,n})_{l^2,T_y}-\mu(\epsilon^{y,x,n},D_x e^{y,n})_{l^2,T_x}(D_x\gamma^{x,n},e^{x,n})_{l^2,T,M}\\
   &+(D_y\gamma^{y,n},e^{y,n})_{l^2,M,T}+(E^{x,n},e^{x,n})_{l^2,T,M}+(E^{y,n},e^{y,n})_{l^2,M,T}).
   \endaligned
   \end{equation}
   Invoking (\ref{forep}) and using the Cauchy-Schwarz inequality results in
   \begin{equation}
   \aligned
   &\|\mathbf{e}^m\|_{l^2}^2 + \mu \sum_{n = 1}^{m} \Delta t \|D\mathbf{e}^n\|^2_{l^2}
   \leq \sum_{n = 0}^{m-1}\Delta t\|\mathbf{e}^m\|_{l^2}^2+e^{\mu T}(1+\mu+\mu^2+\mu^3)O(\|\mathbf{u}\|_{W^{2,\infty}(J;W^{4,\infty}(\Omega))} (h^2+k^2+\Delta t))^2.
   \endaligned
   \end{equation}
   Finally applying Gronwall's inequality, we get that
   \begin{equation}\label{mu1}
   \aligned
   &\|\mathbf{e}^m\|_{l^2}^2 \leq e^{\mu T}(1+\mu+\mu^2+\mu^3)O(\|\mathbf{u}\|_{W^{2,\infty}(J;W^{4,\infty}(\Omega))} (h^2+k^2+\Delta t))^2,
   \endaligned
   \end{equation}
   which leads to the desired result \eqref{e_nonuniform_erroru}. In addition, the proof details for \eqref{e_uniform_erroru} on uniform grids will be proved in the Appendix B for brevity.
   
\end{proof}
\subsection{Error Estimates for the Pressure}
    We now analyze the pressure error estimates in the discrete $l^\infty$ norm in time. 
    \begin{thm}\label{th10}
    	Suppose that the analytical solutions $\textbf{u}$ and $p$ are sufficiently smooth. For the constructed RMAC scheme (\ref{e7})-(\ref{e9}), there exists a positive constant $C$ independent of $h,~k,~\Delta t,$ such that for  $1 \leq m \leq N$
    	\begin{align} \label{e_pressure_nonuniform}
    	\|Z^{m}-p^{m}\|_{l^2}&\leq C(\mathbf{u},\partial_t\mathbf{u},\partial_{tt}\mathbf{u},\partial_{ttt}\mathbf{u},p,T)(h^2+k^2+\Delta t).
    	\end{align}
    	Furthermore, we can obtain the stronger error estimates with $\Delta t = O(h^2)$ on uniform grids,
    	\begin{align} \label{e_pressure_uniform}
    	\|Z^{m}-p^{m}\|_{l^\infty}&\leq C(\mathbf{u},\partial_t\mathbf{u},\partial_{tt}\mathbf{u},\partial_{ttt}\mathbf{u},p,T)(h^2+k^2+\Delta t).
    	\end{align}
    \end{thm}
    
\begin{proof}
	By making the temporal difference quotient of (\ref{e40}), we have
	\begin{equation}\label{e40*}
	\aligned
	&d_td_{t}{e}^{x,n}_{i,j+1/2}-\mu \frac{d_xd_te^{x,n}_{i+1/2,j+1/2}-d_xd_te^{x,n}_{i-1/2,j+1/2}}{h_i}
	-\mu \frac{D_yd_te^{x,n}_{i,j+1}-D_yd_te^{x,n}_{i,j}}{k_{j+1/2}}+d_tD_xe^{p,n}_{i,j+1/2}\\
	=&d_tE^{x,n}_{i,j+1/2}+ \frac{d_t\gamma^{x,n}_{i+1/2,j+1/2}-d_t\gamma^{x,n}_{i-1/2,j+1/2}}{h_i}+\mu \frac{d_t\epsilon^{x,n}_{i+1/2,j+1/2}-d_t\epsilon^{x,n}_{i-1/2,j+1/2}}{h_i}
	+\mu \frac{d_t\epsilon^{x,y,n}_{i,j+1}-d_t\epsilon^{x,y,n}_{i,j}}{k_{j+1/2}}.
	\endaligned
	\end{equation}
	Similarly to (\ref{e43}), we can obtain that
	\begin{equation}\label{e43*}
	\aligned
	&(d_td_te^{x,n},v^{x})_{l^2,T,M}+(d_td_te^{y,n},v^{y})_{l^2,M,T}
	+\mu (d_xd_te^{x,n},d_xv^{x})_{l^2,M}\\
	&+\mu (D_yd_te^{x,n},D_yv^{x})_{l^2,T_y}
	+\mu (d_yd_te^{y,n},d_yv^{y})_{l^2,M}+\mu (D_xd_te^{y,n},D_xv^{y})_{l^2,T_x}\\
	&-(d_te^{p,n},d_xv^{x}+d_yv^{y})_{l^2,M}\\
	=&-\mu(d_t\epsilon^{x,n},d_xv^{x})_{l^2,M}-\mu (d_t\epsilon^{x,y,n},D_yv^{x})_{l^2,T_y}
	-\mu(d_t\epsilon^{y,n},d_yv^{y})_{l^2,M}\\
	&-\mu (d_t\epsilon^{y,x,n},D_xv^{y})_{l^2,T_x}+(D_xd_t\gamma^{x,n},d_te^{x,n})_{l^2,T,M}+(D_yd_t\gamma^{y,n},d_te^{y,n})_{l^2,M,T}\\
	&+(d_tE^{x,n},v^{x})_{l^2,T,M}+(d_tE^{y,n},v^{y})_{l^2,M,T}.
	\endaligned
	\end{equation}
	Setting $\mathbf{v} = d_t \mathbf{e}^n$ in (\ref{e43*}) and summing for $n \geq 2$, we can obtain that
	\begin{equation}
	\aligned
	&\|d_t \mathbf{e}^m\|^2_{l^2} + 2\mu \sum_{n = 2}^{m} \Delta t \|Dd_t\mathbf{e}^n\|^2_{l^2} \\
	&~~\leq \|d_t \mathbf{e}^1\|^2_{l^2} - 2\sum_{n = 2}^{m - 1} \Delta t (e^{p,n},d_xd_td_te^{x,n + 1}+d_yd_td_te^{y,n + 1})_{l^2,M}\\
	&~~~~+2(e^{p,m},d_xd_te^{x,m}+d_yd_te^{y,m})_{l^2,M} - 2(e^{p,1},d_xd_te^{x,2}+d_yd_te^{y,2})\\
	&~~~~-2\mu\sum_{n = 2}^{m}\Delta t((d_t\epsilon^{x,n},d_xd_te^{x,n})_{l^2,M}+ (d_t\epsilon^{x,y,n},D_yd_te^{x,n})_{l^2,T_y}
	+(d_t\epsilon^{y,n},d_yvd_te^{y,n})_{l^2,M}\\
	&~~~~+ (d_t\epsilon^{y,x,n},D_xd_te^{y,n})_{l^2,T_x})
	+2\sum_{n = 2}^{m}\Delta t((d_tE^{x,n},d_te^{x,n})_{l^2,T,M}+(d_tE^{y,n},d_te^{y,n})_{l^2,M,T})\\
	&~~~~+2\sum_{n = 2}^{m}\Delta t((d_tD_x\gamma^{x,n},d_te^{x,n})_{l^2,T,M}+(d_tD_y\gamma^{y,n},d_te^{y,n})_{l^2,M,T}).
	\endaligned
	\end{equation}
	Using Lemma \ref{cross} and (\ref{forep}) and the Cauchy-Schwarz inequality, we have
	\begin{equation}
	\aligned
	&\|d_t \mathbf{e}^m\|^2_{l^2} + 2\mu \sum_{n = 2}^{m} \Delta t \|Dd_t\mathbf{e}^n\|^2_{l^2} \\
	&~~\leq 4\|d_t\mathbf{e}^1\|^2_{l^2} +\sum_{n = 2}^{m - 1} \Delta t \|d_t \mathbf{e}^n\|^2_{l^2} + e^{\mu T}(1+\mu+\mu^2+\mu^3)O(\|\mathbf{u}\|_{W^{3,\infty}(J;W^{4,\infty}(\Omega))} (h^2+k^2+\Delta t))^2.
	\endaligned
	\end{equation}
	Applying Gronwall's inequality gives that
	\begin{equation}
	\aligned
	&\|d_t \mathbf{e}^m\|^2_{l^2}\leq 4e^T\|d_t\mathbf{e}^1\|^2_{l^2} + e^{\mu T}(1+\mu+\mu^2+\mu^3)O(\|\mathbf{u}\|_{W^{3,\infty}(J;W^{4,\infty}(\Omega))} (h^2+k^2+\Delta t))^2.
	\endaligned
	\end{equation}
	As for $\|d_t \mathbf{e}^1\|^2_{l^2}$, setting $\mathbf{v} = d_t\mathbf{e}^n$ in (\ref{e43}) and choosing $\Delta t$ sufficient small, we can easily obtain that
	\begin{equation}
	\|d_t\mathbf{e}^1\|^2_{l^2} \leq (1+\mu+\mu^2)O(\|\mathbf{u}\|_{W^{3,\infty}(J;W^{4,\infty}(\Omega))} (h^2+k^2+\Delta t))^2.
	\end{equation}
Hence
	\begin{equation}\label{4.39}
	\|d_t\mathbf{e}^m\|^2_{l^2} \leq e^{\mu T}(1+\mu+\mu^2+\mu^3)O(\|\mathbf{u}\|_{W^{3,\infty}(J;W^{4,\infty}(\Omega))} (h^2+k^2+\Delta t))^2.
	\end{equation}
	Combining (\ref{4.39}) with (\ref{forep}) leads to
	\begin{equation}\label{eptinfity}
	\aligned
	\|e^{p,m}\|^2_{l^2,M} \leq& e^{\mu T}(1+\mu+\mu^2+\mu^3)O(\|\mathbf{u}\|_{W^{3,\infty}(J;W^{4,\infty}(\Omega))} (h^2+k^2+\Delta t))^2.
	\endaligned
	\end{equation}
	Taking note of Lemma \ref{lem:discrete_invariance}, $\|Z^{m}-p^{m}\|^2_{l^2}$ can be controlled by 
	\begin{equation}
		\aligned
		&\|Z^{m}-p^{m}\|^2_{l^2,M} \\
		&~~\leq \|Z^{(1),m} - 0\|^2_{l^2,M} + \|(P_hp)^{m}-p^{m}\|^2_{l^2,M}\\
		&~~\leq e^{\mu T}(1+\mu+\mu^2+\mu^3)O(\|\mathbf{u}\|_{W^{3,\infty}(J;W^{4,\infty}(\Omega))} (h^2+k^2+\Delta t))^2 + C|(p,1)_{l^2}|^2\\
		&~~\leq e^{\mu T}(1+\mu+\mu^2+\mu^3)O(\|\mathbf{u}\|_{W^{3,\infty}(J;W^{4,\infty}(\Omega))} (h^2+k^2+\Delta t))^2 + O(\|p\|_{L^{\infty}(J;W^{2,\infty}(\Omega))} (h^2+k^2))^2,
		\endaligned
	\end{equation}
	which leads to the desired result \eqref{e_pressure_nonuniform}. In addition, the proof details for \eqref{e_pressure_uniform} on uniform grids will be proved in the Appendix B for brevity.
\end{proof}


\section{Physics-preserving Properties}
In this section, we will prove that the RMAC scheme \eqref{e7}-\eqref{e9} preserves the discrete unconditionally energy-dissipative property, momentum conservation, and angular momentum conservation in the absence of the extra force with the homogeneous boundary condition for the Stokes equations. Furthermore, inspired by \cite{gong2018second, li2020sav}, we construct a fully implicit RMAC scheme for Navier--Stokes equations, rigorously proving that this scheme inherits above discrete unconditionally energy-dissipative property, momentum conservation, and angular momentum conservation.

We define the kinetic energy $E: \mathbf{L}^2(\Omega)\rightarrow \mathbb{R}$, linear momentum $M: \mathbf{L}^2(\Omega)\rightarrow \mathbb{R}^{2}$, and (two-dimensional) angular momentum $M_{\mathbf{x}}: \mathbf{L}^2(\Omega)\rightarrow \mathbb{R}$ by
$$E\left(\mathbf{u}^{*}\right):=\frac{1}{2} \int_{\Omega}\left\lvert\mathbf{u}^{*}\right\rvert^{2} {~d} \mathbf{x},\quad
M\left(\mathbf{u}^{*}\right):=\int_{\Omega} \mathbf{u}^{*} ~{d} \mathbf{x},\quad
M_{\mathbf{x}}\left(\mathbf{u}^{*}\right):=\int_{\Omega} \mathbf{u}^{*} \times \mathbf{x} ~{d} \mathbf{x},$$
for any $\mathbf{u}^{*}=(u_*^1, u_*^2)\in \mathbf{L}^2(\Omega)$, where $\mathbf{x}=(x,y)^\top$ and $\mathbf{u}^{*} \times \mathbf{x}=u_*^1y-u_*^2x$ can be 
regarded as the third component of $(u_*^1, u_*^2,0)^\top\times (x,y,0)^\top$. Similarly to \cite{charnyi2017conservation}, for momentum analysis we assume that the discrete solutions are compactly supported, thereby eliminating the influence of boundary value.

\subsection{Physics-preserving Properties for Stokes Equations}
In this section, we first prove the discrete unconditionally energy-dissipative property, momentum conservation, and angular momentum conservation for the Stokes equations.

\begin{thm}\label{lem_energy}
	The local mass conservation law holds for the constructed RMAC scheme \eqref{e7}-\eqref{e9}. Additionally, in the absence of the extra force $\textbf{g}$, the solutions of MAC scheme \eqref{e7}-\eqref{e9} can satisfy the discrete unconditional energy dissipation law
	\begin{equation}\label{e_energy}
	\aligned
	&  \frac{1}{2}\|\mathbf{W}^{n}\|^2 - \frac{1}{2}\|\mathbf{W}^{n-1}\|^2  \leq - \Delta t \mu \|D\mathbf{W}^n\|^2 \leq 0.
	\endaligned
	\end{equation}
	If we further assume that there only exists a strictly interior subdomain $\hat{\Omega} \subset [x_1,~x_{Nx-1}]\times[y_1,~y_{Ny-1}]$ such that $\textbf{W}$ and $Z$ vanish on $\Omega\setminus\hat{\Omega}$, 
	the solutions of MAC scheme \eqref{e7}-\eqref{e9} can satisfy
	momentum conservation
	\begin{equation}\label{e_momentum}
	\aligned
	( \mathbf{W}^{n} , \mathbf{e}_i  ) = ( \mathbf{W}^{n-1} ,  \mathbf{e}_i ), 
	\endaligned
	\end{equation}
	and  angular momentum conservation
	\begin{equation}\label{e_angular}
	\aligned
	( \mathbf{W}^{n} ,  \hat{\mathbf{x}}) = ( \mathbf{W}^{n-1} ,  \hat{\mathbf{x}} ),
	\endaligned
	\end{equation}
	where we denote by $\mathbf{e}_i \ (i =1,2)$ the unit coordinate vectors in  $x$- and $y$-directions and $\hat{\mathbf{x}} = (y, -x)$.
	\end{thm}

\medskip
\begin{proof}
	Recalling \eqref{e9}, we can easily obtain the local mass conservation law and multiplying \eqref{e7} by $h_i k_{j+1/2} W^{x,n}_{i,j+1/2}$ and summing for $i,j$ with $i=1,\cdots,N_x-1,~j=0,\cdots,N_y-1$, we have 
	\begin{equation}\label{e_energy1}
	\aligned
	&(d_tW^{x,n},W^{x,n})_{l^2,T,M}+\mu (d_xW^{x,n},d_xW^{x,n})_{l^2,M}+\mu (D_yW^{x,n},D_yW^{x,n})_{l^2,T_y}-(Z^{n},d_xW^{x,n})_{l^2,M} \\
	=& \frac{ \| W^{x,n} \|_{l^2,T,M}^2 - \| W^{x,n-1} \|_{l^2,T,M}^2 } { 2 \Delta t} + \frac{ \| W^{x,n} -  W^{x,n-1}   \|_{l^2,T,M}^2  } { 2 \Delta t} + \mu (d_xW^{x,n},d_xW^{x,n})_{l^2,M} \\
	&+\mu (D_yW^{x,n},D_yW^{x,n})_{l^2,T_y}-(Z^{n},d_xW^{x,n})_{l^2,M} = 0.
	\endaligned
	\end{equation}
	
	Similarly in the $y$-direction, we have that
	\begin{equation}\label{e_energy2}
	\aligned
	&(d_tW^{y,n},W^{y,n})_{l^2,M,T}+\mu (d_yW^{y,n},d_yW^{y,n})_{l^2,M}+\mu (D_xW^{y,n},D_xW^{y,n})_{l^2,T_x}-(Z^{n},d_yW^{y,n})_{l^2,M} \\
	=& \frac{ \| W^{y,n} \|_{l^2,M,T}^2 - \| W^{y,n-1} \|_{l^2,M,T}^2 } { 2 \Delta t} + \frac{ \| W^{y,n} -  W^{y,n-1}   \|_{l^2,M,T}^2  } { 2 \Delta t} + \mu (d_yW^{y,n},d_yW^{y,n})_{l^2,M} \\
	&+\mu (D_xW^{y,n},D_xW^{y,n})_{l^2,T_x}-(Z^{n},d_yW^{y,n})_{l^2,M} = 0.
	\endaligned
	\end{equation}
	
	Then combining \eqref{e_energy1} and \eqref{e_energy2}, we can obtain
	\begin{equation}\label{e_energy3}
	\aligned
	& \frac{ \|\mathbf{W}^{n}\|^2 -   \|\mathbf{W}^{n-1}\|^2 }{ 2 \Delta t } \leq -  \frac{ \|\mathbf{W}^{n} - \mathbf{W}^{n-1} \|^2 }{ 2 \Delta t } - \mu \|D\mathbf{W}^n\|^2 \leq 0,
	\endaligned
	\end{equation}
	which leads to the desired energy stability result \eqref{e_energy}. 
	
	Next we shall prove the momentum conservation. Under the assumption that there only exists a strictly interior subdomain $\hat{\Omega} \subset [x_1,~x_{Nx-1}]\times[y_1,~y_{Ny-1}]$ such that $\textbf{W}$ and $Z$ vanish on $\Omega\setminus\hat{\Omega}$ and in the absence of the extra force $\textbf{g}$, it is reasonable to get the boundary condition of the differential compact support that
	\begin{figure}
		\centering
		\includegraphics[width=0.4\linewidth]{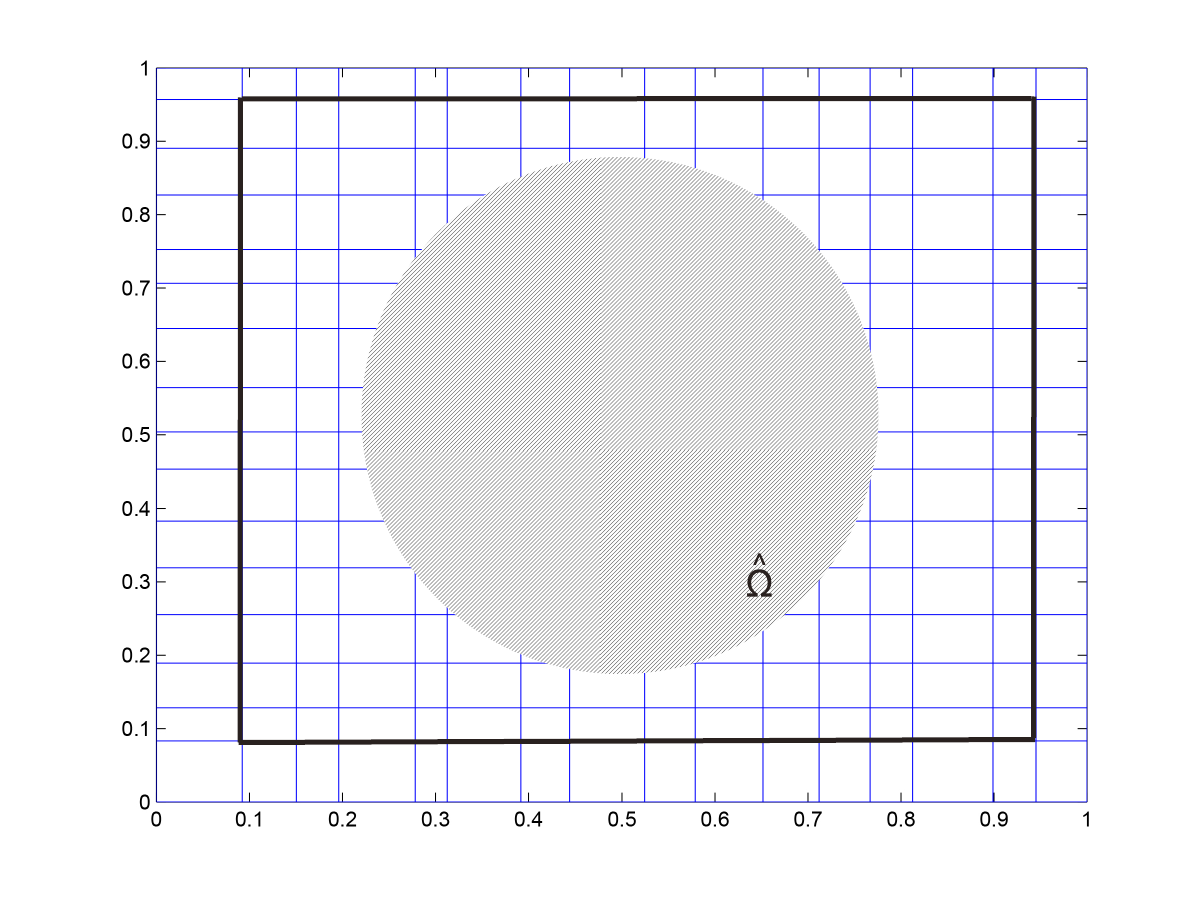}
		\caption{The differential compact support}
	\end{figure}
	\begin{equation}\label{e66}
	\left\{
	\begin{array}{l}
	\displaystyle W_{0,j+1/2}^{x,n}=W_{1,j+1/2}^{x,n}=W_{N_x-1,j+1/2}^{x,n}=W_{N_x,j+1/2}^{x,n}=0,\quad 0\leq j\leq N_y-1,\\
	\displaystyle W_{i,0}^{x,n}=W_{i,1/2}^{x,n}=W_{i,N_y-1/2}^{x,n}=W_{i,N_y}^{x,n}=0,~~~~~~~~~\quad 0\leq i\leq N_x,\\
	\displaystyle  W_{0,j}^{y,n}=W_{1/2,j}^{y,n}=W_{N_x-1/2,j}^{y,n}=W_{N_x,j}^{y,n}=0,~~~~~~~~~\quad 0\leq j\leq N_y,\\
	\displaystyle W_{i+1/2,0}^{y,n}=W_{i+1/2,1}^{y,n}=W_{i+1/2,N_y-1}^{y,n}=W_{i+1/2,N_y}^{y,n}=0,~\quad 0\leq i\leq N_x-1,\\
	Z^n_{1/2, j+1/2} = Z^n_{i+1/2, 1/2} = 0.
	\end{array}
	\right.
	\end{equation}
	Multiplying \eqref{e7} by $h_i k_{j+1/2} \times 1$ and summing for $i,j$ with $i=1,\cdots,N_x-1,~j=0,\cdots,N_y-1$, and multiplying \eqref{e8} by $ h_{i+1/2} k_j \times 1$ and summing for $i,j$ with $i=0,\cdots,N_x-1,~j=1,\cdots,N_y-1$, we have
	\begin{equation}\label{e_momentum1}
	\aligned
	& ( \mathbf{W}^{n} , \mathbf{e}_i) = ( \mathbf{W}^{n-1} ,  \mathbf{e}_i ),\quad i=1,2.
	\endaligned
	\end{equation}
	Similarly, multiplying \eqref{e7} by $h_i k_{j+1/2} (y_{j+1/2})$ and summing for $i,j$ with $i=1,\cdots,N_x-1,~j=0,\cdots,N_y-1$, and multiplying \eqref{e8} by $ h_{i+1/2} k_j (-x_{i+1/2})$ and summing for $i,j$ with $i=0,\cdots,N_x-1,~j=1,\cdots,N_y-1$, we have
	\begin{equation}\label{e_angular1}
	\aligned
	& ( \mathbf{W}^{n} , \hat{\mathbf{x}} ) = ( \mathbf{W}^{n-1} , \hat{\mathbf{x}}) .
	\endaligned
	\end{equation}
\end{proof}

\subsection{Extension to Navier--Stokes equations}
We all know that the Navier--Stokes equations have only more nonlinear term $\mathbf{u} \cdot \nabla \mathbf{u}$ than the Stokes equation. The properties of other terms have been obtained in the above discussion. Next, we mainly focus on the nonlinear term. Firstly, we construct the following fully implicit RMAC scheme: We find $\{W^{x,n}_{i,j+1/2}\},~\{W^{y,n}_{i+1/2,j}\}$ and $\{Z^{n}_{i+1/2,j+1/2}\}$ for $n \geq 1$ such that, 
\begin{align}
&d_{t}{W}^{x,n}_{i,j+1/2}-\mu D_x(d_xW^{x,n})_{i,j+1/2}
-\mu d_y(D_yW^{x,n})_{i,j+1/2} + \frac{1}{2}[ W^{x,n}D_x(\mathcal{P}_h^xW^{x,n})+\mathcal{P}^x_hd_x((W^{x,n})^2) \notag\\
&~~+\mathcal{P}^y_h(\mathcal{P}^x_hW^{y,n}D_yW^{x,n}) + d_y(\mathcal{P}^y_hW^{x,n}\mathcal{P}^x_hW^{y,n})]_{i,j+1/2}+D_xZ_{i,j+1/2}^n=f_{i,j+1/2}^{x,n},\label{e7.}\\
&d_{t}{W}^{y,n}_{i+1/2,j}-\mu d_x(D_xW^{y,n})_{i+1/2,j}
-\mu D_y(d_yW^{y,n})_{i+1/2,j}+\frac{1}{2}[W^{y,n}D_y(\mathcal{P}^y_hW^{y,n}) + \mathcal{P}^y_hd_y((W^{y,n})^2)\notag\\
&~~+\mathcal{P}^x_h(\mathcal{P}^y_hW^{x,n}D_xW^{y,n}) + d_x(\mathcal{P}^y_hW^{x,n}\mathcal{P}^x_hW^{y,n}) ]_{i+1/2,j}+D_yZ_{i+1/2,j}^n=f_{i+1/2,j}^{y,n},\label{e8.}\\
&d_xW^{x,n}_{i+1/2,j+1/2}+d_yW^{y,n}_{i+1/2,j+1/2}=0,\label{e9.}
\end{align}
where $\mathcal{P}^x_h$ and $\mathcal{P}^y_h$ are linear interpolation operators in the $x$- and $y$-directions.
\begin{thm}\label{lem_energy*}
	Under the conditions of Theorem \ref{lem_energy}, the solutions of the RMAC scheme \eqref{e7.}-\eqref{e9.} can satisfy the discrete unconditional energy dissipation law
	\begin{equation}\label{e_energy*}
	\aligned
	&  \|\mathbf{W}^{n}\|^2 - \|\mathbf{W}^{n-1}\|^2 \leq - \Delta t \mu \|D\mathbf{W}^n\|^2 \leq 0,
	\endaligned
	\end{equation}
	and momentum conservation
	\begin{equation}\label{e_momentum*}
	\aligned
	( \mathbf{W}^{n} ,\mathbf{e}_i ) = ( \mathbf{W}^{n-1} ,  \mathbf{e}_i ), 
	\endaligned
	\end{equation}
	and  angular momentum conservation
	\begin{equation}\label{e_angular*}
	\aligned
	( \mathbf{W}^{n} ,\hat{\mathbf{x}} ) = ( \mathbf{W}^{n-1} , \hat{\mathbf{x}} ),
	\endaligned
	\end{equation}
	with the same $\mathbf{e}_i \ (i =1,2)$ and $\hat{\mathbf{x}}$ as in Theorem~\ref{lem_energy}.
\end{thm}
\begin{proof}
	We mainly focus on the nonlinear term and set
	\begin{equation}
	    \aligned
	    \alpha^{x,n}_{i,j+1/2} &= \frac{1}{2}[ W^{x,n}D_x(\mathcal{P}_h^xW^{x,n})+\mathcal{P}^x_hd_x((W^{x,n})^2) \\
	    &~~+ \mathcal{P}^y_h(\mathcal{P}^x_hW^{y,n}D_yW^{x,n}) + d_y(\mathcal{P}^y_hW^{x,n}\mathcal{P}^x_hW^{y,n})]_{i,j+1/2},\\
	    \alpha^{y,n}_{i+1/2,j} &= \frac{1}{2}[W^{y,n}D_y(\mathcal{P}^y_hW^{y,n}) + \mathcal{P}^y_hd_y((W^{y,n})^2)\\
	    &~~+\mathcal{P}^x_h(\mathcal{P}^y_hW^{x,n}D_xW^{y,n}) + d_x(\mathcal{P}^y_hW^{x,n}\mathcal{P}^x_hW^{y,n}) ]_{i+1/2,j}.
	    \endaligned
	\end{equation}
	By using the similar energy estimates in \cite{li2020sav}, we have 
	\begin{equation}
		(\alpha^n,\mathbf{W}^n) = 0,
	\end{equation}
	thus we can obtain (\ref{e_energy*}).
	
	For momentum conservation, using Lemma \ref{le2} and (\ref{e66}), we can obtain that
	\begin{equation}\label{M1}
		\aligned
		(W^{x,n}D_x(\mathcal{P}^x_hW^{x,n}),1)_{l^2,T,M} = (D_x(\mathcal{P}^x_hW^{x,n}),W^{x,n})_{l^2,T,M} = -(\mathcal{P}^x_hW^{x,n},d_xW^{x,n})_{l^2,M},
		\endaligned
	\end{equation}
	\begin{equation}\label{M2}
	\aligned
	(\mathcal{P}^x_hd_x((W^{x,n})^2),1)_{l^2,T,M} = (d_x((W^{x,n})^2),1)_{l^2,M} &= 0,\\
	(d_y(\mathcal{P}^y_hW^{x,n}\mathcal{P}^x_hW^{y,n}),1)_{l^2,T,M} &= 0,
	\endaligned
	\end{equation}
	\begin{equation}\label{M3}
		\aligned
		&(\mathcal{P}^y_h(\mathcal{P}^x_hW^{y,n}D_yW^{x,n}),1)_{l^2,T,M} = (\mathcal{P}^x_hW^{y,n}D_yW^{x,n},1)_{l^2,T} = (D_yW^{x,n},\mathcal{P}^x_hW^{y,n})_{l^2,T} \\
		&~~=-(W^{x,n},\mathcal{P}^x_hd_yW^{y,n})_{l^2,T,M} = -(\mathcal{P}^x_hW^{x,n},d_yW^{y,n})_{l^2,M}. 
		\endaligned
	\end{equation}
	Then combining (\ref{M1})-(\ref{M3}), we have that
	\begin{equation}
		(\alpha^{x,n},1)_{l^2,T,M} = -(\mathcal{P}^x_hW^{x,n},d_xW^{x,n}+d_yW^{y,n})_{l^2,M} = 0.
	\end{equation}
	Similarly in the $y$-direction, we get that
	\begin{equation}
		( \alpha^{n} ,\mathbf{e}_i ) = 0,
	\end{equation}
	thus we can obtain (\ref{e_momentum*}).
	
	For angular momentum conservation, using Lemma \ref{le2} and boundary condition (\ref{e66}), we have that
	\begin{align}\label{A1}
	(W^{x,n}D_x(\mathcal{P}^x_hW^{x,n}),y)_{l^2,T,M} & = -(\mathcal{P}^x_hW^{x,n},yd_xW^{x,n})_{l^2,M}=-(W^{x,n},y\mathcal{P}^x_hd_xW^{x,n})_{l^2,T,M},\\
	\label{A2}
	(\mathcal{P}^x_hd_x((W^{x,n})^2),y)_{l^2,T,M} & = (d_x((W^{x,n})^2),y)_{l^2,M} = 0,\\
	\nonumber
	(\mathcal{P}^y_h(\mathcal{P}^x_hW^{y,n}D_yW^{x,n}),y)_{l^2,T,M} & = (\mathcal{P}^x_hW^{y,n}D_yW^{x,n},\mathcal{P}_h^yy)_{l^2,T} = (D_yW^{x,n},y\mathcal{P}^x_hW^{y,n})_{l^2,T} \\
	\label{A3} =-(W^{x,n},d_y(y\mathcal{P}^x_hW^{y,n}))&_{l^2,T,M}  = -(\mathcal{P}^x_hW^{x,n},\mathcal{P}^y_hW^{y,n})_{l^2,M} - (W^{x,n},y\mathcal{P}^x_hd_yW^{y,n})_{l^2,T,M},\\
	\label{A4}
		(d_y(\mathcal{P}^y_hW^{x,n}\mathcal{P}^x_hW^{y,n}),y)_{l^2,T,M} &= -(\mathcal{P}^y_hW^{x,n}\mathcal{P}^x_hW^{y,n},1)_{l^2,T}=-(\mathcal{P}^x_hW^{x,n},\mathcal{P}^y_hW^{y,n})_{l^2,M}.
	\end{align}
	Then combining (\ref{A1})-(\ref{A4}), we have that
	\begin{equation}\label{AA1}
	(\alpha^{x,n},y)_{l^2,T,M} = -(\mathcal{P}^x_hW^{x,n},\mathcal{P}^y_hW^{y,n})_{l^2,M}.
	\end{equation}
    Similarly in the $y$-direction, we get that
    \begin{equation}\label{AA2}
    (\alpha^{y,n},-x)_{l^2,M,T} = (\mathcal{P}^x_hW^{x,n},\mathcal{P}^y_hW^{y,n})_{l^2,M}.
    \end{equation}
    Therefore we have
    \begin{equation}
    	( \alpha^{n} ,\hat{\mathbf{x}}) = 0,
    \end{equation}
    which leads to the desired result (\ref{e_angular*}).
\end{proof}

\begin{figure}[htbp]
	\centering
	\subfloat[]
	{\includegraphics[width=0.4\linewidth]{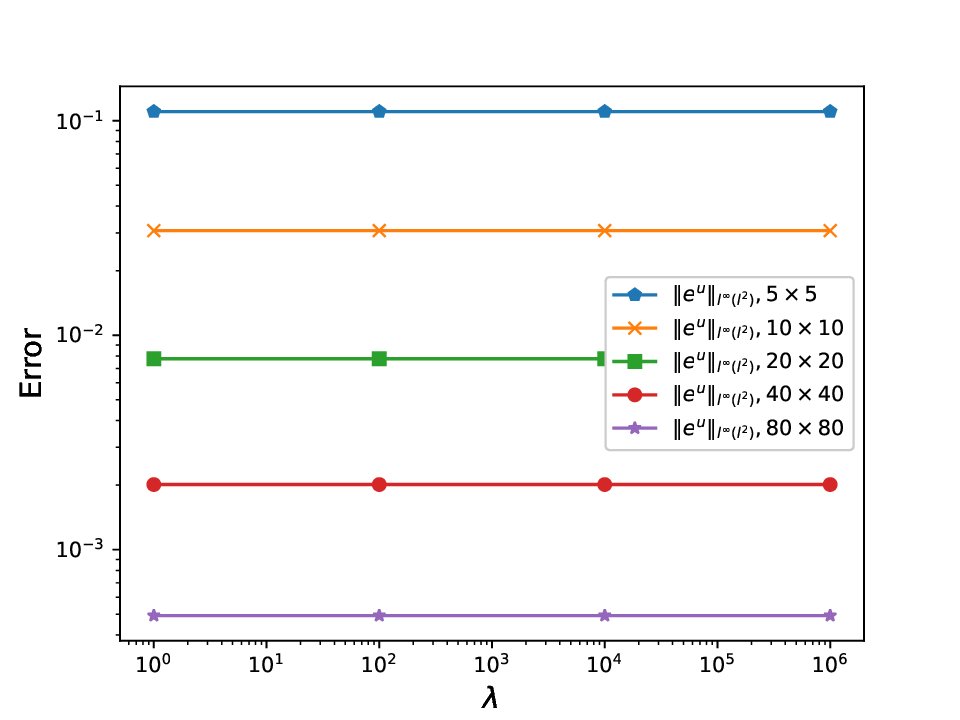}}
	\subfloat[]
	{\includegraphics[width=0.4\linewidth]{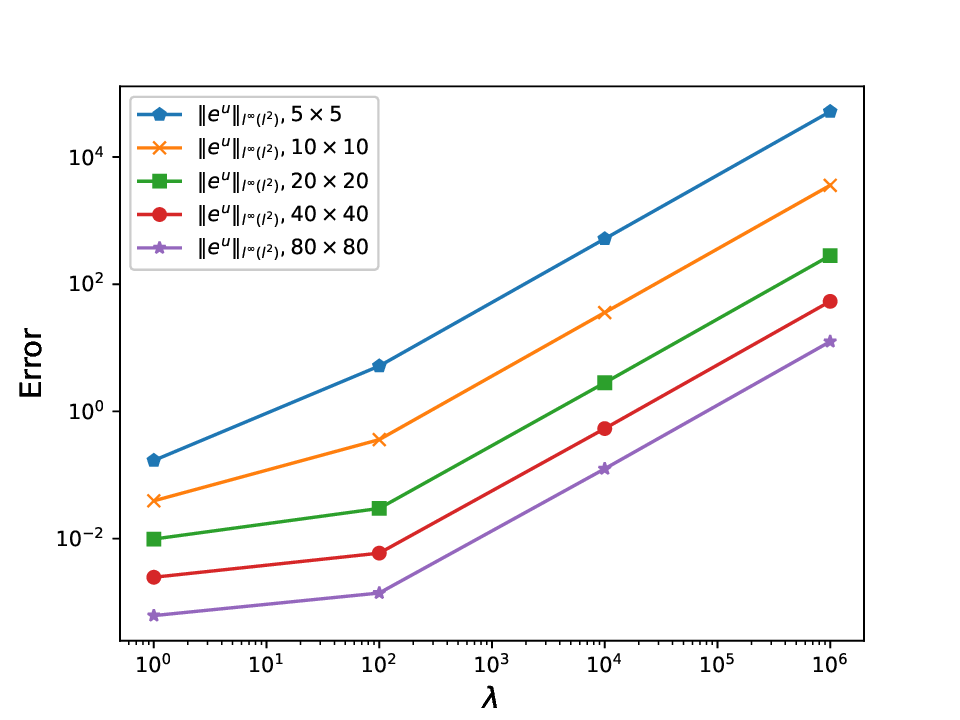}}\\
	\subfloat[]
	{\includegraphics[width=0.4\linewidth]{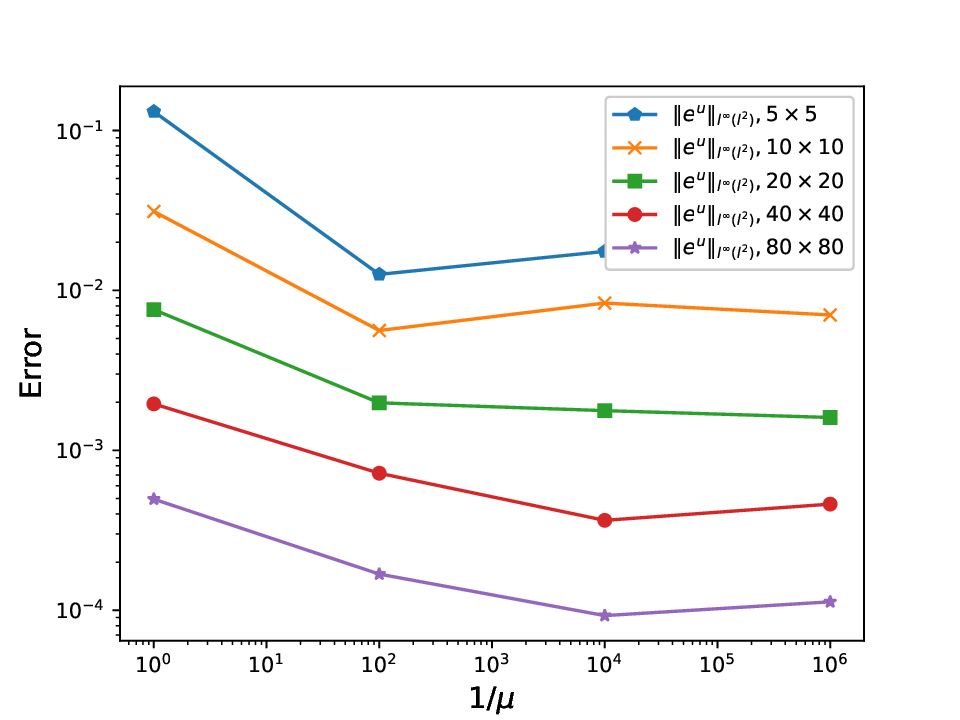}}
	\subfloat[]
	{\includegraphics[width=0.4\linewidth]{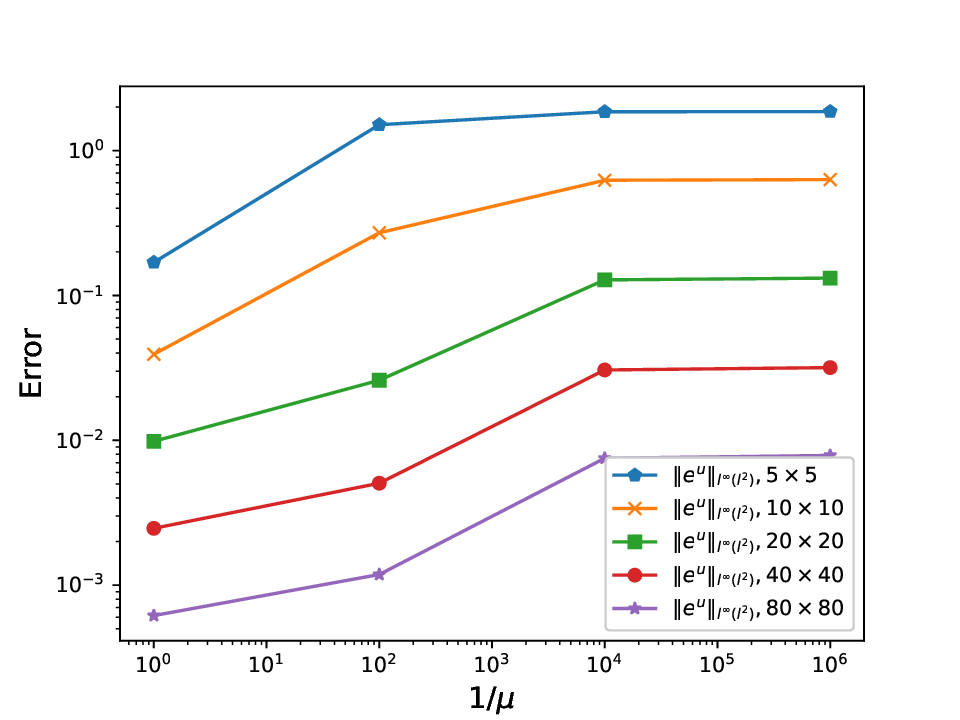}}
	\caption{Robust-error profiles for Stokes equations: RMAC (Left) vs. MAC (Right) schemes on non-uniform grids (Example 1)}
	\label{fig1}
\end{figure}
\begin{figure}[htbp]
	\centering
	\subfloat[]
	{\includegraphics[width=0.4\linewidth]{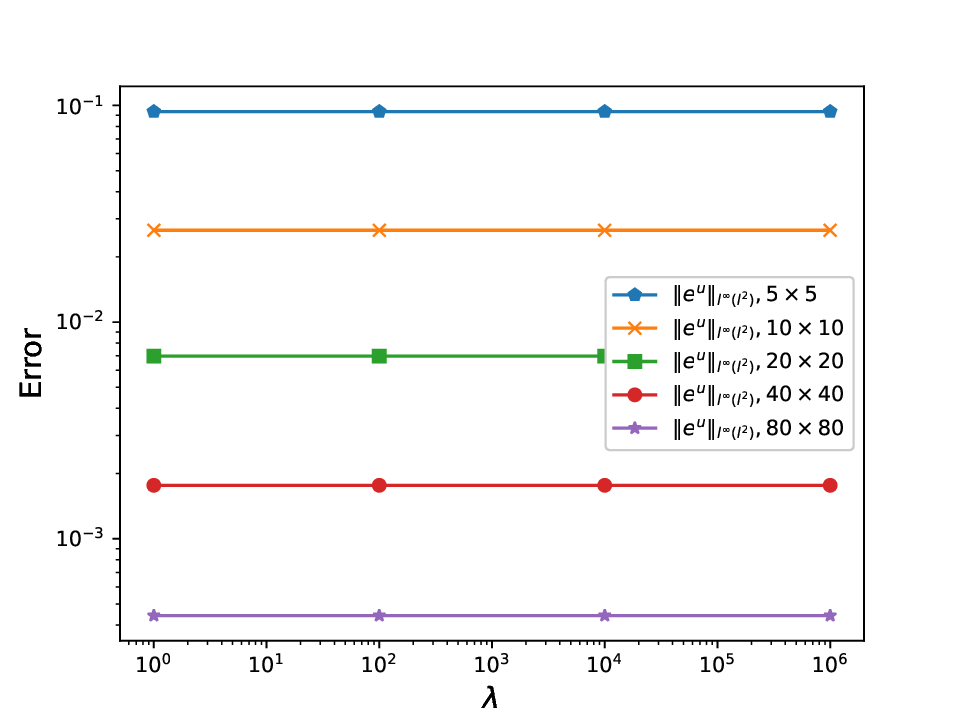}}
	\subfloat[]
	{\includegraphics[width=0.4\linewidth]{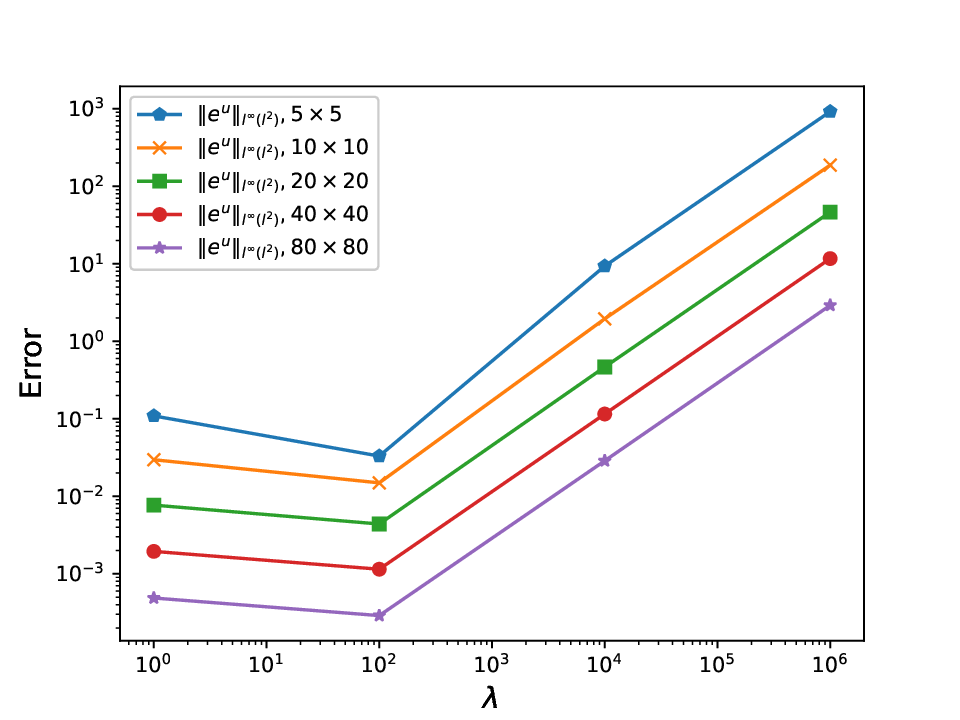}}\\
	\subfloat[]
	{\includegraphics[width=0.4\linewidth]{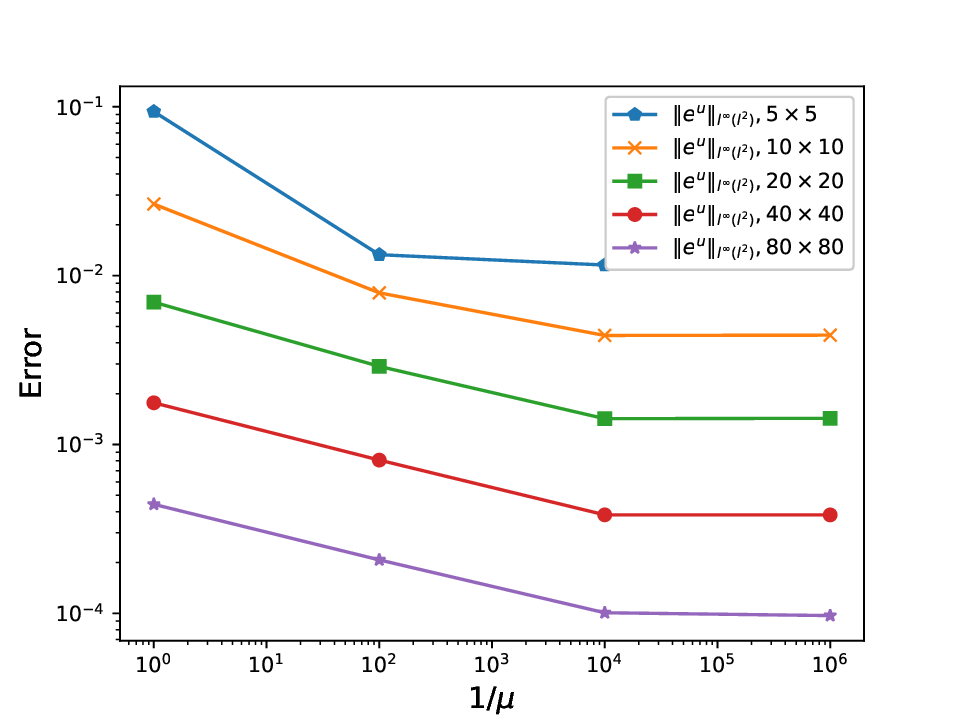}}
	\subfloat[]
	{\includegraphics[width=0.4\linewidth]{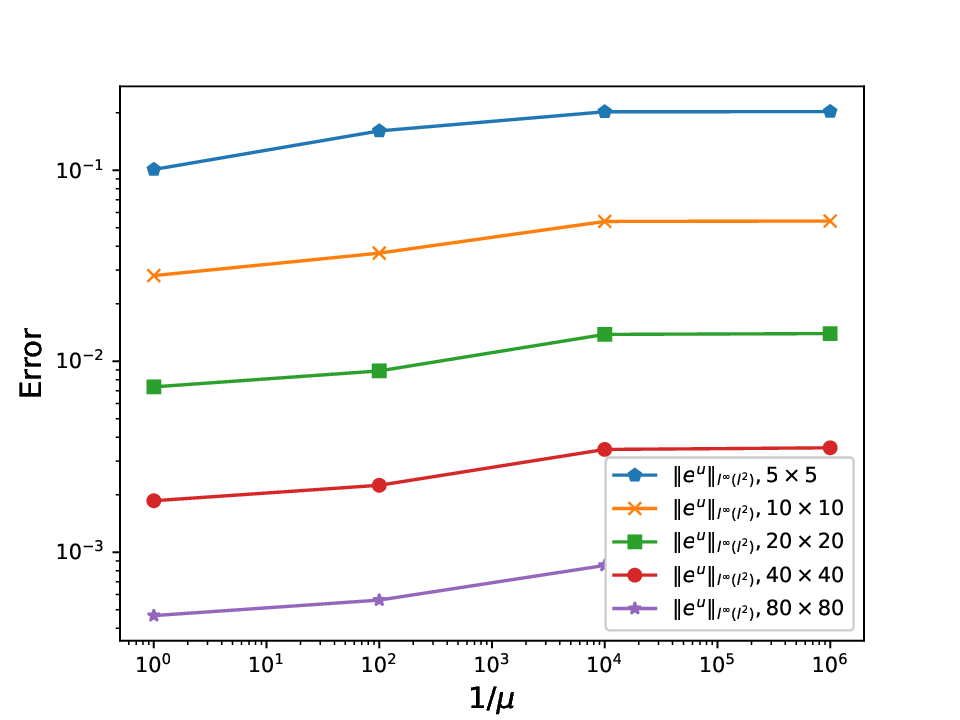}}
	\caption{Robust-error profiles for Navier-Stokes equations: RMAC (Left) vs. MAC (Right) schemes on uniform grids (Example 2)}
	\label{fig3}
\end{figure}
\begin{table}[htbp]
	\centering
	\caption{Error and convergence rates of RMAC for Example 1 on non-uniform grids}
	
	\label{table3}
	\small
	\begin{tabular}{|c|c|c|c|c|} \hline
		$N_x\times N_y$&$||e^p||_{l^\infty(l^2,M)}$ &$Rate$&$||e^u||_{l^{\infty}(l^2)}$&$Rate$  \\ \hline
		$5\times 5$ &2.87E-1  &---       &1.10E-1      &---           \\
		$10\times 10$ &8.83E-2 &1.92     &3.07E-2      &2.08         \\
		$20\times 20$ &2.11E-2  &2.06    &7.77E-3      &1.98           \\
		$40\times 40$ &6.01E-3  &1.82     &2.01E-3      &1.97          \\
		$80\times 80$ &1.43E-3  &2.09    &4.93E-4      
		&2.05          \\
		\hline
	\end{tabular}
\end{table}

\begin{table}[htbp]
	\centering
	\caption{Error and convergence rates of RMAC for Example 2 on uniform grids}
	
	\label{table4}
	\small
	\begin{tabular}{|c|c|c|c|c|} \hline
		$N_x\times N_y$&$||e^p||_{l^\infty(l^\infty)}$ &$Rate$&$||e^u||_{l^{\infty}(l^\infty)}$&$Rate$  \\ \hline
		$5\times 5$ &1.61E-1  &---       &9.36E-2      &---           \\
		$10\times 10$ &6.15E-2  &1.39     &2.65E-2      &1.82         \\
		$20\times 20$ &1.79E-2  &1.78    &6.96E-3     &1.93           \\
		$40\times 40$ &4.68E-3  &1.93     &1.77E-3      &1.98          \\
		$80\times 80$ &1.18E-3  &1.98    &4.42E-4      
		&2.00           \\
		\hline
	\end{tabular}
\end{table}

%

\section{Numerical Examples}
In this section, we present numerical experiments employing the constructed RMAC scheme and the classical MAC scheme  in \cite{Rui2016,li2018superconvergence} on both non-uniform and uniform spatial discretizations, with the computational domain fixed as the unit square $\Omega=(0,1)\times(0,1)$ for simplicity.

We test Example 1 for the Stokes equations and Example 2 for the Navier--Stokes equations by using the constructed RMAC scheme and the classical MAC scheme to verify the pressure-robustness and $\mu$-robustness for the velocity error. In addition we also give the corresponding errors for the pressure. The initial spatial partition is a $8 \times 8$ grid. The time step is refined as $\Delta t = 1/N^2_x = 1/N^2_y$ to show convergence rates both in time and space. The pressure-robustness is verified by adjusting $\lambda$, and the $\mu$-robustness is verified by reducing $\mu$. The value of $\lambda$ is taken as 1, $1e2$, $1e4$, $1e6$ with $\mu = 1$, and the value of $\mu$ is taken as 1, $1e-2$, $1e-4$, $1e-6$ with $\lambda = 1$.

\textbf{Example 1 for the Stokes equations}: The right hand side of the equation is computed according to the analytic solution given below.
\begin{equation*}
\aligned
\begin{cases}p(x,y)=\lambda e^tsin^3(4 \pi x)sin^3(4 \pi y),\\
u^x(x,y)=\pi e^t sin^2(\pi x)sin(2\pi y),\\
u^y(x,y)=-\pi e^t sin(2\pi x)sin^2(\pi y).
\end{cases}
\endaligned
\end{equation*}

The numerical results for the Stokes equation are listed in Fig. \ref{fig1} and Table \ref{table3} with non-uniform grids for the RMAC scheme and MAC scheme. The constructed non-uniform grids are characterized by ratios between the maximum and minimum mesh sizes exceeding 1.5 across the computational domain. 

As we can see, the errors of velocity for the constructed RMAC scheme maintain optimal convergence rates under steep pressure gradients, demonstrating its pressure-robust property, as numerically validated in Fig. \ref{fig1}(a). In contrast, the classical MAC scheme in \cite{Rui2016,li2018superconvergence} demonstrates evident velocity error degradation in Fig. \ref{fig1}(b), where the pressure induce error amplification under steep gradients, which is caused by its inherent lack of pressure-robustness. We also present the $\mu$-robustness of the error of velocity  in Fig. \ref{fig1}(c) and Fig. \ref{fig1}(d) on non-uniform grids for the RMAC scheme and MAC scheme, respectively. While both schemes demonstrate $\mu$-robustness in Figs. \ref{fig1}(c) and \ref{fig1}(d), the RMAC scheme exhibits that its velocity errors undergo monotonic reduction with increasing $\mu$ before attaining error stability, whereas the MAC scheme manifests progressive error amplification under $\mu$-ascending conditions until reaching equilibrium. In addition, we can see that all the convergence and superconvergence results for the RMAC scheme with $\lambda = 1,~\mu = 1$ on non-uniform grids are verified  in Table \ref{table3}. 


\textbf{Example 2 for the Navier-Stokes equations}: The right hand side of the equation is computed according to the analytic solution {\color{black}given below}.
\begin{equation*}
\aligned
\begin{cases}p(x,y,t)= 10\lambda e^t((x-0.5)^3y^2 + (1-x)^3(y-0.5)^3),\\
u^x(x,y,t)=-256tx^2(x-1)^2y(y-1)(2y-1),\\
u^y(x,y,t)=-u^x(y,x,t).
\end{cases}
\endaligned
\end{equation*}
The numerical results for the  Navier-Stokes equations  are listed in Fig. \ref{fig3} and Table \ref{table4} with uniform grids for the RMAC scheme and MAC scheme. We present the pressure-robustness of the error of velocity in Fig. \ref{fig3}(a) and Fig. \ref{fig3}(b) for the RMAC scheme and MAC scheme, respectively. And we demonstrate the $\mu$-robustness of the error of velocity  from Fig. \ref{fig3}(c) for RMAC scheme and Fig. \ref{fig3}(d) for MAC scheme with uniform grids, respectively. As we can see,  the performance on uniform grids is more stable than that on non-uniform grids. In addition, we can see that all the convergence results for the RMAC scheme with $\lambda = 1,~\mu = 1$ on uniform grids are verified  in Table \ref{table4}. 

\section{Conclusion}
In this work, we construct RMAC schemes  for solving time-dependent Stokes and Navier-Stokes equations on non-uniform grids, delivering three main contributions: i) a duality-based framework for the pressure estimation, which serves as a counterpart to the discrete LBB condition;  ii) pressure-robustness and $\mu$-robustness for the RMAC scheme, ensuring velocity errors remain independent of pressure approximations and bounded under $\mu \ll 1$; iii) rigorous second-order superconvergence non-uniform grids for the velocity and pressure for the Stokes equation. Furthermore, we also prove that the constructed RMAC schemes preserve local mass conservation, unconditional energy dissipation, momentum conservation and angular momentum conservation for the Stokes and Navier-Stokes equations. Numerical validations confirm theoretical convergence rates and robustness. These results can be extended to three dimensional problems.
\appendix
\section{Proof of consistency lemma}
In this appendix we derive that the auxiliary function statisfy the discrete equations to a high order of accuracy in Lemma \ref{lemcon}.
We first estimate the time difference quotient for $\widetilde{u}^{x,n}$ and $\widetilde{u}^{y,n}$.
\begin{lem}\label{le9}
	Under the condition of Lemma \ref{lemcon}, there holds
	\begin{equation}
	\aligned
	&d_{t}{\widetilde{u}}^{x,n}_{i,j+1/2}-\frac{1}{h_i}\int_{x_{i-1/2}}^{x_{i+1/2}}\frac{\partial u^{x,n}}{\partial t}(x,y_{j+1/2})dx \\
	&~~~~~~~= -\frac{\gamma^{x,n}_{i+1/2,j+1/2} - \gamma^{x,n}_{i-1/2,j+1/2}}{h_i} + O(\|\mathbf{u}\|_{W^{2,\infty}(J;W^{2,\infty}(\Omega))}(h^2+k^2 + \Delta t)), \\
	&d_{t}{\widetilde{u}}^{y,n}_{i+1/2,j}-\frac{1}{k_j}\int_{y_{j-1/2}}^{y_{j+1/2}}\frac{\partial u^{y,n}}{\partial t}(x_{i+1/2},y)dy \\
	&~~~~~~~= -\frac{\gamma^{y,n}_{i+1/2,j+1/2} - \gamma^{y,n}_{i+1/2,j-1/2}}{k_j} + O(\|\mathbf{u}\|_{W^{2,\infty}(J;W^{2,\infty}(\Omega))}(h^2 + k^2 + \Delta t)),
	\endaligned
	\end{equation}
	where $ \displaystyle \gamma^{x,n}_{i+1/2,j+1/2} = \frac{h^2_{i+1/2}}{8}\frac{\partial^2 u^{x,n}_{i+1/2,j+1/2}}{\partial t \partial x},~\gamma^{y,n}_{i+1/2,j+1/2} = \frac{k^2_{j+1/2}}{8}\frac{\partial^2 u^{y,n}_{i+1/2,j+1/2}}{\partial t \partial y}$.
\end{lem}
\begin{proof}
	By Talor's expansion, we obtain that
	\begin{equation}\label{5.7}
	\aligned
	d_{t}{\widetilde{u}}^{x,n}_{i,j+1/2} = \frac{\partial u^{x,n}_{i,j+1/2}}{\partial t} + O(\|\mathbf{u}\|_{W^{2,\infty}(J;W^{2,\infty}(\Omega))} (k^2 + \Delta t)).
	\endaligned
	\end{equation}
	In addtion, we have 
	\begin{equation}\label{5.8}
	\aligned
	&\frac{1}{h_i}\int_{x_{i-1/2}}^{x_{i+1/2}}\frac{\partial u^{x,n}}{\partial t}(x,y_{j+1/2})dx  \\
	&= \frac{\partial u^{x,n}_{i,j+1/2}}{\partial t}+ \frac{1}{8h_i}(\frac{\partial^2 u^{x,n}_{i+1/2,j+1/2}}{\partial t\partial x}h^2_{i+1/2} - \frac{\partial^2 u^{x,n}_{i-1/2,j+1/2}}{\partial t\partial x}h^2_{i-1/2}) \\
	&~~~~+ O(\|\mathbf{u}\|_{W^{1,\infty}(J;W^{2,\infty}(\Omega))}h^2).
	\endaligned
	\end{equation}
	Combining (\ref{5.7}) and (\ref{5.8}) we obtain that
	\begin{equation}
	\aligned
	&d_{t}{\widetilde{u}}^{x,n}_{i,j+1/2}-\frac{1}{h_i}\int_{x_{i-1/2}}^{x_{i+1/2}}\frac{\partial u^{x,n}}{\partial t}(x,y_{j+1/2})dx \\
	&~~~~~~~= -\frac{\gamma^{x,n}_{i+1/2,j+1/2} - \gamma^{x,n}_{i-1/2,j+1/2}}{h_i} + O(\|\mathbf{u}\|_{W^{2,\infty}(J;W^{2,\infty}(\Omega))}(h^2+k^2 + \Delta t)).
	\endaligned
	\end{equation}
	Similarly, we can have the same conclusion for $y$-direction.
\end{proof}

Now we estimate $\frac{d_x\widetilde{u}^{x,n}_{i+1/2,j+1/2} - d_x\widetilde{u}^{x,n}_{i-1/2,j+1/2}}{h_i}$
and $\frac{d_y\widetilde{u}^{y,n}_{i+1/2,j+1/2} - d_y\widetilde{u}^{y,n}_{i+1/2,j-1/2}}{k_j}$. For $n \geq 1$ set
\begin{equation}\label{e38}
\left\{
\begin{array}{l}
\displaystyle
\epsilon^{x,x,n}_{i+1/2,j+1/2}=\frac{h^2_{i+1/2}}{24}\frac{\partial^3 u^{x,n}_{i+1/2,j+1/2}}{\partial x^3}-
\frac{k^2_{j+1/2}}{8}\frac{\partial^3 u^{x,n}_{i+1/2,j+1/2}}{\partial x \partial y^2},\\
\displaystyle
\epsilon^{y,y,n}_{i+1/2,j+1/2}=-\frac{h^2_{i+1/2}}{8}\frac{\partial^3 u^{y,n}_{i+1/2,j+1/2}}{\partial x^2 \partial y}+
\frac{k^2_{j+1/2}}{24}\frac{\partial^3 u^{y,n}_{i+1/2,j+1/2}}{\partial y^3}.\\
\end{array}
\right.
\end{equation}

\begin{lem}\label{le7}
	Under the condition of Lemma \ref{lemcon}, there holds
	\begin{equation}\label{e139}
	\aligned
	&\frac{d_x\widetilde{u}^{x,n}_{i+1/2,j+1/2} - d_x\widetilde{u}^{x,n}_{i-1/2,j+1/2}}{h_i} - \frac{1}{h_i}\int_{x_{i - 1/2}}^{x_{i + 1/2}}\frac{\partial^2 u^{x,n}}{\partial x^2}(x,y_{j + 1/2})dx \\
	&~~~~~~~~= \frac{\epsilon^{x,x,n}_{i+1/2,j+1/2} - \epsilon^{x,x,n}_{i-1/2,j+1/2}}{h_i} 
	+ O(\|\mathbf{u}\|_{L^{\infty}(J;W^{4,\infty}(\Omega))}(h^2 + k^2)),\\
	&\frac{d_y\widetilde{u}^{y,n}_{i+1/2,j+1/2} - d_y\widetilde{u}^{y,n}_{i+1/2,j-1/2}}{k_j} - \frac{1}{k_j} \int_{y_{j-1/2}}^{y_{j+1/2}}\frac{\partial^2 u^{y,n}}{\partial y^2}(x_{i+1/2},y) dy \\
	&~~~~~~~~= \frac{\epsilon^{y,y,n}_{i+1/2,j+1/2} - \epsilon^{y,y,n}_{i+1/2,j-1/2}}{k_j} + O(\|\mathbf{u}\|_{L^{\infty}(J;W^{4,\infty}(\Omega))}(h^2 + k^2)).
	\endaligned
	\end{equation}
\end{lem}
\begin{proof}
	From (\ref{e25}) we have that
	\begin{equation}
	\aligned
	d_x &\widetilde{u}^{x,n}_{i+1/2,j+1/2} \\
	&= \frac{\partial u^x_{i+1/2,j+1/2}}{\partial x} + \frac{h^2_{i+1/2}}{24}\frac{\partial^3 u^{x,n}_{i+1/2,j+1/2}}{\partial x^3} + \frac{k^2_{j+1/2}}{24}\frac{\partial^3 u^{x,n}_{i+1/2,j+1/2}}{\partial x \partial y^2}\\
	&~~~~~~-
	\frac{k^2_{j+1/2}}{6}\frac{\partial^3 u^{x,n}_{i+1/2,j+1/2}}{\partial x \partial y^2}+ O(\|\mathbf{u}\|_{L^\infty(J;W^{4,\infty}(\Omega))}(h^3 + k^2h + k^3))\\
	&=\frac{\partial u^x_{i+1/2,j+1/2}}{\partial x} + \frac{h^2_{i+1/2}}{24}\frac{\partial^3 u^{x,n}_{i+1/2,j+1/2}}{\partial x^3} -
	\frac{k^2_{j+1/2}}{8}\frac{\partial^3 u^{x,n}_{i+1/2,j+1/2}}{\partial x \partial y^2}\\
	&~~~~~~+ O(\|\mathbf{u}\|_{L^\infty(J;W^{4,\infty}(\Omega))}(h^3 + k^2h + k^3)),
	\endaligned
	\end{equation}
	thus we obtain that
	\begin{equation}
	\aligned
	&\frac{d_x\widetilde{u}^{x,n}_{i+1/2,j+1/2} - d_x\widetilde{u}^{x,n}_{i-1/2,j+1/2}}{h_i} - \frac{1}{h_i}\int_{x_{i - 1/2}}^{x_{i + 1/2}}\frac{\partial^2 u^{x,n}}{\partial x^2}(x,y_{j + 1/2})dx \\
	&~~~~~~~~= \frac{\epsilon^{x,x,n}_{i+1/2,j+1/2} - \epsilon^{x,x,n}_{i-1/2,j+1/2}}{h_i} 
	+ O(\|\mathbf{u}\|_{L^{\infty}(J;W^{4,\infty}(\Omega))}(h^2 + k^2)).
	\endaligned
	\end{equation}
	Similarly, we can have the same conclusion for $y$-direction.
\end{proof}

Next we estimate $\frac{D_y\widetilde{u}^{x,n}_{i,j+1} - D_y\widetilde{u}^{x,n}_{i,j}}{k_{j + 1/2}}$
and $\frac{D_x\widetilde{u}^{y,n}_{i+1,j} - D_x\widetilde{u}^{y,n}_{i,j}}{h_{i + 1/2}}$.
Set for $n\geq1$,
\begin{equation}\label{e111}
\left\{
\begin{array}{l}
\displaystyle\epsilon^{x,y,n}_{i,j+1}=-\frac{k^3_{j+3/2}+k^3_{j+1/2}}{24k_{j+1}}\frac{\partial^3 u^{x,n}_{i,j+1}}{\partial y^3},~~~1\leq i\leq N_x-1,~1\leq j\leq N_y-1,\\
\displaystyle\epsilon^{x,y,n}_{i,0}=-\frac{k^2_{1/2}}{12}\frac{\partial^3 u^{x,n}_{i,0}}{\partial y^3},~
\epsilon^{x,y,n}_{i,N_y}=-\frac{k^2_{N_y-1/2}}{12}\frac{\partial^3 u^{x,n}_{i,N_y}}{\partial y^3},
~~~1\leq i\leq N_x-1,\\
\displaystyle\epsilon^{y,x,n}_{i+1,j}=-\frac{h^3_{i+3/2}+h^3_{i+1/2}}{24h_{i+1}}\frac{\partial^3 u^{y,n}_{i+1,j}}{\partial x^3},~~~1\leq i\leq N_x-1,~1\leq j\leq N_y-1,\\
\displaystyle\epsilon^{y,x,n}_{0,j}=-\frac{h^2_{1/2}}{12}\frac{\partial^3 u^{y,n}_{0,j}}{\partial x^3},~
\epsilon^{y,x,n}_{N_x,j}=-\frac{h^2_{N_x-1/2}}{12}\frac{\partial^3 u^{y,n}_{N_x,j}}{\partial x^3},
~~~1\leq j\leq N_y-1.\\
\end{array}
\right.
\end{equation}
\begin{lem}\label{le8}
	Under the condition of Lemma \ref{lemcon}, then there holds
	\begin{equation}\label{e112}
	\aligned
	&\frac{D_y\widetilde{u}^{x,n}_{i,j+1} - D_y\widetilde{u}^{x,n}_{i,j}}{k_{j+1/2}} - \frac{1}{h_i}\int_{x_{i-1/2}}^{x_{i+1/2}}\frac{\partial^2 u^{x,n}}{\partial y^2}(x,y_{j+1/2})dx \\ 
	&~~~~~~~~= -\frac{\delta^{x,n}_{i+1/2,j+1/2} - \delta^{x,n}_{i-1/2,j+1/2}}{h_i} + \frac{\epsilon^{x,y,n}_{i,j+1} - \epsilon^{x,y,n}_{i,j}}{k_{j + 1/2}} + O(\|\mathbf{u}\|_{L^{\infty}(J;W^{4,\infty})}(h^2+k^2)),\\
	&\frac{D_x\widetilde{u}^{y,n}_{i+1,j} - D_x\widetilde{u}^{y,n}_{i,j}}{h_{i+1/2}} - \frac{1}{k_j}\int_{y_{j-1/2}}^{y_{j+1/2}}\frac{\partial^2 u^{y,n}}{\partial x^2}(x_{i+1/2},y)dy \\
	&~~~~~~~~= -\frac{\delta^{y,n}_{i+1/2,j+1/2} - \delta^{y,n}_{i+1/2,j-1/2}}{k_j} + \frac{\epsilon^{y,x,n}_{i+1,j} - \epsilon^{y,x,n}_{i,j}}{h_{i+1/2}}+O(\|\mathbf{u}\|_{L^{\infty}(J;W^{4,\infty})}(h^2+k^2)),
	\endaligned
	\end{equation}
	where $ \displaystyle \delta_{i+1/2,j+1/2}^{x,n} = \frac{h^2_{i+1/2}}{8}\frac{\partial^3 u^{x,n}_{i+1/2,j+1/2}}{\partial y^2 \partial x}, ~\delta_{i+1/2,j+1/2}^{y,n} = \frac{k^2_{j+1/2}}{8}\frac{\partial^3 u^{y,n}_{i+1/2,j+1/2}}{\partial x^2 \partial y}$.
\end{lem}
\begin{proof}
	First we consider that
	\begin{equation}
	\aligned
	\frac{D_y\widetilde{u}^{x,n}_{i,j+1} - D_y\widetilde{u}^{x,n}_{i,j}}{k_{j+1/2}} - \frac{1}{k_{j+1/2}}\int_{y_j}^{y_{j+1}}\frac{\partial^2 u^{x,n}}{\partial y^2}(x_i,y)dy.
	\endaligned
	\end{equation}
	By Talor's expansion it is clear that for $(x_i,y_{j+1}) \in \Omega$
	\begin{equation}\label{le8e1}
	\aligned
	D_y\widetilde{u}^{x,n}_{i,j+1} = &\frac{1}{k_{j+1}} (\widetilde{u}^{x,n}_{i,j+3/2} - \widetilde{u}^{x,n}_{i,j+1/2}) \\
	&= \frac{\partial u^{x,n}_{i,j+1}}{\partial y} + \frac{\partial^2 u^{x,n}_{i,j+1}}{\partial y^2} \frac{k_{j+3/2} - k_{j+1/2}}{3} + \frac{\partial^3 u^{x,n}_{i,j+1}}{\partial y^3} \frac{k^3_{j+3/2}+k^3_{j+1/2}}{24k_{j+1}} \\
	&~~~~~~-\frac{\partial^2 u^{x,n}_{i,j+1}}{\partial y^2} \frac{k_{j+3/2} - k_{j+1/2}}{3} - \frac{\partial^3 u^{x,n}_{i,j+1}}{\partial y^3} \frac{k^3_{j+3/2}+k^3_{j+1/2}}{12k_{j+1}}\\
	&~~~~~~+O(\|\mathbf{u}\|_{L^\infty(J;W^{4,\infty}(\Omega))}k^3)\\
	&= \frac{\partial u^{x,n}_{i,j+1}}{\partial y} - \frac{\partial^3 u^{x,n}_{i,j+1}}{\partial y^3} \frac{k^3_{j+3/2}+k^3_{j+1/2}}{24k_{j+1}} + O(\|\mathbf{u}\|_{L^\infty(J;W^{4,\infty}(\Omega))}k^3).
	\endaligned
	\end{equation}
	To consider the case for $j=0$, we have to estimate $D_y\widetilde{u}^{x,n}_{i,0}$. Since $\displaystyle k_0 = \frac{k_{1/2}}{2}$, we have
	\begin{equation}\label{le8e3}
	\aligned
	D_y\widetilde{u}^{x,n}_{i,0} =& \frac{1}{k_0} \left(\widetilde{u}^{x,n}_{i,1/2} - u^{x,n}_{i,0}\right) \\
	=& \frac{1}{k_0}\left(k_0\frac{\partial u^{x,n}_{i,0}}{\partial y}-\frac{k_0^3}{3}\frac{\partial^3 u^{x,n}_{i,0}}{\partial y^3} + O(\|\mathbf{u}\|_{L^\infty(J;W^{4,\infty}(\Omega))}k^4)\right) \\
	=& \frac{\partial u^{x,n}_{i,0}}{\partial y}-\frac{k_{1/2}^2}{12}\frac{\partial^3 u^{x,n}_{i,0}}{\partial y^3} + O(\|\mathbf{u}\|_{L^\infty(J;W^{4,\infty}(\Omega))}k^3).
	\endaligned
	\end{equation}
	Similarly to (\ref{le8e3}), we obtain for $j = N_y$
	\begin{equation}
	D_y\widetilde{u}^{x,n}_{i,N_y} = \frac{\partial u^{x,n}_{i,N_y}}{\partial y}- \frac{k_{N_y-1/2}^2}{12}\frac{\partial^3 u^{x,n}_{i,N_y}}{\partial y^3} + O(\|\mathbf{u}\|_{L^\infty(J;W^{4,\infty}(\Omega))}k^3).
	\end{equation}
	Therefore, we have that
	\begin{equation}\label{lemD1}
	\aligned
	&\frac{D_y\widetilde{u}^{x,n}_{i,j+1} - D_y\widetilde{u}^{x,n}_{i,j}}{k_{j+1/2}} - \frac{1}{k_{j+1/2}}\int_{y_j}^{y_{j+1}}\frac{\partial^2 u^{x,n}}{\partial y^2}(x_i,y)dy \\
	&~~= \frac{\epsilon^{x,y,n}_{i,j+1} - \epsilon^{x,y,n}_{i,j}}{k_{j + 1/2}} + O(\|\mathbf{u}\|_{L^{\infty}(J;W^{4,\infty})}k^2).
	\endaligned
	\end{equation}
	Then we consider that 
	\begin{equation}
	\aligned
	\frac{1}{k_{j+1/2}}\int_{y_j}^{y_{j+1}}\frac{\partial^2 u^{x,n}}{\partial y^2}(x_i,y)dy - \frac{1}{h_i}\int_{x_{i-1/2}}^{x_{i+1/2}}\frac{\partial^2 u^{x,n}}{\partial y^2}(x,y_{j+1/2})dx.
	\endaligned
	\end{equation}
	It is clear that
	\begin{equation}\label{5.21}
	\frac{1}{k_{j+1/2}}\int_{y_j}^{y_{j+1}}\frac{\partial^2 u^{x,n}}{\partial y^2}(x_i,y)dy = \frac{\partial^2 u^{x,n}_{i,j+1/2}}{\partial y^2} + O(\|\mathbf{u}\|_{L^\infty(J;W^{4,\infty}(\Omega))}k^2).
	\end{equation}
	And we have that by Talor's expansion
	\begin{equation}
	\aligned
	\frac{\partial^2 u^{x,n}}{\partial y^2}(x,y_{j+1/2}) =  \frac{\partial^2 u^{x,n}_{i,j+1/2}}{\partial y^2} + \frac{\partial^3 u^{x,n}_{i,j+1/2}}{\partial y^2\partial x}(x - x_i) + O(\|\mathbf{u}\|_{L^\infty(J;W^{4,\infty}(\Omega))}h^2), 
	\endaligned
	\end{equation}
	then
	\begin{equation}\label{5.23}
	\aligned
	\frac{1}{h_i}&\int_{x_{i-1/2}}^{x_{i+1/2}}\frac{\partial^2 u^{x,n}}{\partial y^2}(x,y_{j+1/2})dx \\
	&= \frac{\partial^2 u^{x,n}_{i,j+1/2}}{\partial y^2} + \frac{1}{8h_i}(\frac{\partial^3 u^{x,n}_{i,j+1/2}}{\partial y^2\partial x}h^2_{i+1/2} - \frac{\partial^3 u^{x,n}_{i,j+1/2}}{\partial y^2\partial x}h^2_{i-1/2}) + O(\|\mathbf{u}\|_{L^\infty(J;W^{4,\infty}(\Omega))}h^2)\\
	&= \frac{\partial^2 u^{x,n}_{i,j+1/2}}{\partial y^2} + \frac{1}{8h_i}(\frac{\partial^3 u^{x,n}_{i+ 1/2,j+1/2}}{\partial y^2\partial x}h^2_{i+1/2} - \frac{\partial^3 u^{x,n}_{i-1/2,j+1/2}}{\partial y^2\partial x}h^2_{i-1/2}) + O(\|\mathbf{u}\|_{L^\infty(J;W^{4,\infty}(\Omega))}h^2)\\
	&=\frac{\partial^2 u^{x,n}_{i,j+1/2}}{\partial y^2} + \frac{\delta^{x,n}_{i+1/2,j+1/2} - \delta^{x,n}_{i-1/2,j+1/2}}{h_i} + O(\|\mathbf{u}\|_{L^\infty(J;W^{4,\infty}(\Omega))}h^2)
	\endaligned
	\end{equation}
	Combining (\ref{5.21}) and (\ref{5.23}) we obatin that
	\begin{equation}
	\aligned
	&\frac{1}{k_{j+1/2}}\int_{y_j}^{y_{j+1}}\frac{\partial^2 u^{x,n}}{\partial y^2}(x_i,y)dy - \frac{1}{h_i}\int_{x_{i-1/2}}^{x_{i+1/2}}\frac{\partial^2 u^{x,n}}{\partial y^2}(x,y_{j+1/2})dx\\
	&~~~~~~~~= -\frac{\delta^{x,n}_{i+1/2,j+1/2} - \delta^{x,n}_{i-1/2,j+1/2}}{h_i} + O(\|\mathbf{u}\|_{L^\infty(J;W^{4,\infty}(\Omega))}(h^2 + k^2)),
	\endaligned
	\end{equation}
	thus we have that
	\begin{equation}
	\aligned
	&\frac{D_y\widetilde{u}^{x,n}_{i,j+1} - D_y\widetilde{u}^{x,n}_{i,j}}{k_{j+1/2}} - \frac{1}{h_i}\int_{x_{i-1/2}}^{x_{i+1/2}}\frac{\partial^2 u^{x,n}}{\partial y^2}(x,y_{j+1/2})dx \\ 
	&~~~~~~~~= -\frac{\delta^{x,n}_{i+1/2,j+1/2} - \delta^{x,n}_{i-1/2,j+1/2}}{h_i} + \frac{\epsilon^{x,y,n}_{i,j+1} - \epsilon^{x,y,n}_{i,j}}{k_{j + 1/2}} + O(\|\mathbf{u}\|_{L^{\infty}(J;W^{4,\infty})}(h^2+k^2)).
	\endaligned
	\end{equation}
	Similarly, the same conclusion can be obtained for $y$-direction.
\end{proof}
\section{ $l^{\infty}(l^{\infty})$ error analysis on uniform grids}
Inspired by \cite{dong2023second,hou1993second,strang1964accurate}, we use the Strang's trick to obtain the main lemma of high-order consistency on uniform grids. 
\begin{lem}\label{final}
	Assume the solution of Stokes equations is sufficiently smooth. There exist error expansions $\hat{\mathbf{u}}$ and $\hat{p}$, which are $O(h^2)$ perturbations of $\mathbf{u}$ (the exact solusion) with $\Delta t = O(h^2)$:
	 \begin{equation}
	 	\aligned
	 	\hat{\mathbf{u}} &= \mathbf{u} + h^2\tilde{\mathbf{u}}_1 + h^4\tilde{\mathbf{u}}_2, \\
	 	\hat{p} &= h^2\tilde{p}_1 + h^4\tilde{p}_2,
	 	\endaligned
	 \end{equation}
	 where the function $\tilde{\mathbf{u}}$ and $\tilde{p}$ are smooth and their derivatives can be bounded in terms of $\mathbf{u}$ and its derivatives. These functions satisfy the discrete equations to a high order of accuracy in the following sense:
	 \begin{align}
	 &d_{t}{\hat{u}}^{x,n}_{i,j+1/2}-\mu \frac{d_x\hat{u}^{x,n}_{i+1/2,j+1/2}-d_x\hat{u}^{x,n}_{i-1/2,j+1/2}}{h_i}
	 -\mu \frac{D_y\hat{u}^{x,n}_{i,j+1}-D_y\hat{u}^{x,n}_{i,j}}{k_{j+1/2}}\notag\\
	 &~~~~~+D_x\hat{p}^{n}_{i,j+1/2}=\hat{\beta}^{x,n}_{i,j+1/2} + O(h^6),\ \ 1\leq i\leq N_x-1,0\leq j\leq N_y-1,\\
	 &d_{t}{\hat{u}}^{y,n}_{i+1/2,j}-\mu \frac{D_x\hat{u}^{y,n}_{i+1,j}-D_x\hat{u}^{y,n}_{i,j}}{h_{i+1/2}}
	 -\mu \frac{d_y\hat{u}^{y,n}_{i+1/2,j+1/2}-d_y\hat{u}^{y,n}_{i+1/2,j-1/2}}{k_{j}}\notag\\
	 &~~~~~+D_y\hat{p}^{n}_{i+1/2,j}=\hat{\beta}^{y,n}_{i+1/2,j} + O(h^6),\ \ 0\leq i\leq N_x-1,1\leq j\leq N_y-1,\\
	 &d_x\hat{u}^{x,n}_{i+1/2,j+1/2}+d_y\hat{u}^{y,n}_{i+1/2,j+1/2}=O(h^4),\ \ 0\leq i\leq N_x-1,0\leq j\leq N_y-1, \\
	 &\hat{u}^{x,n}_{i,-1/2} + \hat{u}^{x,n}_{i,1/2} = O(h^6),~~\hat{u}^{x,n}_{i,N_y+1/2} + \hat{u}^{x,n}_{i,N_y-1/2} = O(h^6)\\
	 &\hat{u}^{y,n}_{-1/2,j} + \hat{u}^{y,n}_{1/2,j} = O(h^6),~~\hat{u}^{y,n}_{N_x+1/2,j} + \hat{u}^{y,n}_{N_y-1/2,j} = O(h^6),
	 \end{align}
	 with $\hat{\mathbf{u}}^0 = 0$.
\end{lem}

By combining the consistency lemma \ref{final} with the preceding non-uniform grids error analysis and inverse inequalities, we rigorously derive the spatial $l^\infty$ norm estimates on uniform grids.
\bibliographystyle{siamplain}
\bibliography{Stokes_Robustness}

\end{document}